\documentclass[11pt]{amsart}

\usepackage{amsmath,amssymb,amsthm,amsfonts,amscd,flafter,epsf,epsfig,graphicx,verbatim}
\usepackage[all]{xy}

\title[Khovanov homology, open books, and tight contact structures]{Khovanov homology, open books, and tight contact structures}

\author[John A.\ Baldwin]{John A.\ Baldwin}
\address{Department of Mathematics,
Princeton University,
Priceton, NJ 08544-1000, USA}
\email{baldwinj@math.princeton.edu}
\thanks{}

\author[Olga Plamenevskaya]{Olga Plamenevskaya}
\address{Department of Mathematics,
Stony Brook University,
Stony Brook, NY 11794-3651, USA}
\email{olga@math.sunysb.edu}
\thanks{The first author was partially supported by an NSF Postdoctoral Fellowship.}

\newcommand\Sc{\text{Spin}^c}
\newcommand\spc{\mathfrak{s}}

\newcommand\so{\mathfrak{s}_o}
\newcommand\kh{\widetilde{Kh}}

\newcommand\hf{\widehat{HF}}

\newcommand\cf{\widehat{CF}}

\newcommand\hfk{\widehat{HFK}}
\newcommand\zz{\mathbb{Z}}

\newcommand\zzt{\mathbb{Z}_2}

\newcommand\khc{\psi}
\newcommand\khct{\widetilde{\psi}}

\newcommand{\veem}{\mathbf{v}_-}
\newcommand{\veep}{\mathbf{v}_+}
\newcommand{\QQ}{\mathbb{Q}}

\newtheorem{theorem}{Theorem}[section]
\newtheorem{lemma}[theorem]{Lemma}

\newtheorem{corollary}[theorem]{Corollary}
\newtheorem{proposition}[theorem]{Proposition}

\theoremstyle{definition}
\newtheorem{definition}[theorem]{Definition}

\newtheorem{remark}[theorem]{Remark}

\newtheorem{example}[theorem]{Example}

\begin{document}
\begin{abstract}  
We define the reduced Khovanov homology of an open book $(S,\phi)$, and we identify a distinguished ``contact element" in this group which may be used to establish the tightness or non-fillability of contact structures compatible with $(S,\phi)$. Our construction generalizes the relationship between the reduced Khovanov homology of a link and the Heegaard Floer homology of its branched double cover. As an application, we give combinatorial proofs of tightness for several contact structures which are not Stein-fillable. Lastly, we investigate a comultiplication structure on the reduced Khovanov homology of an open book which parallels the comultiplication on Heegaard Floer homology defined in \cite{bald3}.
\end{abstract} 

\maketitle

\section{Introduction}
\label{sec:intro}
The goal of this paper is to demonstrate how Khovanov homology and related ideas may be used to combinatorially establish the tightness or non-fillability of certain contact structures. Let $S$ be a compact, oriented surface with boundary, and let $\phi$ be a composition of Dehn twists around homotopically non-trivial curves in $S$. The abstract open book $(S,\phi)$ corresponds to a contact 3-manifold, which we denote by $(M_{S,\phi},\,\xi_{S,\phi})$ \cite{giroux,thwin}. In this paper, we use the link surgeries spectral sequence machinery of Ozsv{\'a}th and Szab{\'o} \cite{osz12} to define a filtered chain complex $(C(S,\phi), D)$ whose homology is isomorphic to $\hf(-M_{S,\phi})$ (we work with $\zzt$ coefficients throughout). We then define the \emph{reduced Khovanov homology} of the open book $(S,\phi)$ to be the $E^2$ term of the spectral sequence associated to this filtered complex, and we identify an element $\khc(S,\phi) \in \kh(S,\phi)$ which is related to the Ozsv{\'a}th-Szab{\'o} contact invariant $c(S,\phi) \in \hf(-M_{S,\phi})$ via this spectral sequence. 

Let $S_{k,r}$ denote the genus $k$ surface with $r$ boundary components. By Giroux's correspondence \cite{giroux}, every contact 3-manifold is compatible with an open book of the form $(S_{k,1},\phi).$ Moreover, any boundary-fixing diffeomorphism of $S_{k,1}$ is isotopic (rel. $\partial S_{k,1}$) to a composition of Dehn twists around the curves $\alpha_0, \dots, \alpha_{2k}$ depicted in Figure \ref{fig:genset} \cite{hum}. We show that $\kh(S_{k,1},\phi)$ and $\khc(S_{k,1},\phi)$ are combinatorially computable when $\phi$ is such a composition.

\begin{figure}[!htbp]
\begin{center}
\includegraphics[width=7cm]{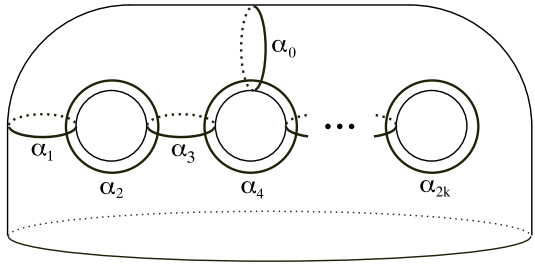}
\label{fig:genset}
\caption{\quad The surface $S_{k,1}$, and the curves $\alpha_0,\dots, \alpha_{2k}$.}
\end{center}
\end{figure}

When the contact manifold $(M_{S,\phi},\,\xi_{S,\phi})$ is the branched double cover of a transverse link, our construction specializes to the setup of \cite{pla1}. Indeed, let $w = w(\sigma_1,\sigma_1^{-1},\dots,\sigma_{2k},\sigma_{2k}^{-1})$ be a word in the elementary generators (and their inverses) of the braid group on $2k+1$ strands, and let $L_w$ denote the closure of the braid specified by $w$. We may think of $L_w$ as a transverse link in standard contact structure $\xi_{st}$ on $S^3$, and lift $\xi_{st}$ to a contact structure $\xi_{L_w}$ on the double cover $\Sigma(L_w)$ of $S^3$ branched along $L_w$. The contact structure $\xi_{L_w}$ is compatible with a natural open book decomposition $(S_{k,1},\phi_w)$ of $\Sigma(L_w)$, where $\phi_w = w(D_{\alpha_1},D_{\alpha_1}^{-1},\dots,D_{\alpha_{2k}},D_{\alpha_{2k}}^{-1})$ (here, $D_{\gamma}$ stands for the right-handed Dehn twist around the curve $\gamma$). According to \cite{osz12}, the Heegaard Floer homology $\hf(-\Sigma(L_w))$ can be computed via a link surgeries spectral sequence whose $E^2$ term is isomorphic to the reduced Khovanov homology of $L_w$. In this case, $\kh(S_{k,1},\phi_w)$ is the same as $\kh(L_w),$ and the element $\khc(S_{k,1},\phi_w)$ coincides with the transverse link invariant $\khc(L_w)$ introduced by the second author in \cite{pla1}. As shown by L. Roberts in \cite{lrob1}, $\khc(L_w)$ ``corresponds" (in a sense to be made precise) to the contact invariant $c(\xi_{L_w}) \in \hf(-\Sigma(L_w))$. Our construction generalizes this result.

The vector space $\kh(S,\phi)$ inherits a grading from the filtration of $(C(S,\phi),D)$, which we refer to as the ``homological grading" (and also as the ``$I$-grading") since it agrees with the homological grading on reduced Khovanov homology under the specialization described in the previous paragraph. The element $\khc(S,\phi)$ is contained in homological grading $0$. In Section \ref{sec:invc}, we show that the graded vector space $\kh(S,\phi)$ is invariant under stabilization of the open book. The element $\khc(S,\phi)\in \kh(S,\phi)$ is invariant under positive stabilization, and is killed by negative stabilization. On the other hand, $\kh(S,\phi)$ is not an invariant of the isotopy class of $\phi$ (the rank of $\kh(S,\phi)$ depends on the precise way that $\phi$ is written as a composition of Dehn twists). 
Even so, the element $\khc(S,\phi)$ may sometimes be used, per the following theorem, to determine whether the contact structure $\xi_{S,\phi}$ is tight.

\begin{theorem}
\label{thm:tight}
If the spectral sequence from $\kh(S,\phi)$ to $\hf(-M_{S,\phi})$ collapses at the $E^2$ term then $\khc(S,\phi) \neq 0$ implies that $c(S,\phi) \neq 0$, and, hence, that $\xi_{S,\phi}$ is tight. 
\end{theorem}

Observe that this spectral sequence collapses at the $E^2$ term whenever $$\text{rk}(\kh(S,\phi)) = |H_1(-M_{S,\phi};\mathbb{Z})|,$$ since $\text{rk}(\hf(-M_{S,\phi})) \geq |H_1(-M_{S,\phi};\mathbb{Z})|$ \cite{osz14}.
Thus, in favorable cases, the collapsing condition is easy to verify (we use the computer programs {\tt Kh} and {\tt Trans} \cite{baldurl} to compute $\text{rk}(\kh(S,\phi))$ and to determine whether $\khc(S,\phi)\neq 0$). In the special case of branched double covers, we can often do without a computer since the spectral sequence from $\kh(L)$ to $\hf(-\Sigma(L))$ collapses at the $E^2$ term as long as $L$ is a quasi-alternating link \cite{osz12}. In particular, if the transverse knot $K$ belongs to a quasi-alternating knot type, then Theorem \ref{thm:tight} implies that $c(\xi_{K}) \neq 0$ whenever $\khc(K) \neq 0$, a result conjectured in \cite{pla1}. 

\begin{theorem} 
\label{thm:nonzero} 
If $K$ is a transverse knot for which $sl(K)=s(K)-1$ then $\khc(K) \neq 0$. The converse is 
also true if $K$ belongs to a quasi-alternating (or any $\kh$-thin\footnote{A knot is said to be ``$\kh$-thin" if its reduced Khovanov homology is supported in bi-gradings $(i,j)$, where $j-2i$ is some fixed constant (see \cite{manozs}, for example).}) knot type.
Here, $sl(K)$ is the self-linking number of the transverse knot, and $s(K)$ is Rasmussen's invariant \cite{ras3}. 
\end{theorem}

\begin{corollary} 
\label{cor:tightness} If $K$ is a transverse representative of a quasi-alternating knot and $sl(K) = \sigma(K)-1$ then $c(\xi_K) \neq 0$, and, hence, $\xi_K$ is tight. Here, $\sigma$ is the knot signature (with the convention that the right-handed trefoil has signature 2). 
\end{corollary}

Note that $sl(K) \leq s(K)-1$ for any transverse knot $K$ \cite{pla3, sh}. Therefore, the hypothesis in Theorem \ref{thm:nonzero} is equivalent to the sharpness of this upper bound for the self-linking number. In \cite{ng}, Ng tabulates the maximal self-linking numbers for knots with at most 10 crossings. Combining those values with the results above, we can, in certain cases, establish the existence of a tight contact structure on the branched double cover of a given knot. 

In Subsection \ref{ssec:tightness}, we provide several examples which demonstrate that our Khovanov-homological machinery is indeed useful and efficient for proving tightness. In particular, we show that $\xi_K$ is tight but not Stein-fillable (our tightness result is therefore non-trivial) for several infinite families of transverse knots $K$ which satisfy the hypotheses of Corollary \ref{cor:tightness}. Some of the knots we consider are transverse 3-braids. When $K$ is a transverse 3-braid, the contact structure $\xi_K$ is compatible with a genus one, one boundary component open book, and the question of whether $\xi_K$ is tight or overtwisted is resolved in \cite{bald1, hkm3}. However, in contrast to the techniques used in \cite{bald1, hkm3}, our methods apply 
to transverse braids of arbitary braid index, require no explicit calculation of the Heegaard Floer contact invariant, and are completely combinatorial. 

Our construction may also be used to combinatorially prove that certain contact structures are not strongly symplectically fillable.

\begin{proposition}
\label{prop:vanishing}
If $\kh(S,\phi)$ is supported in non-positive homological gradings and $\khc(S,\phi)=0$, then $c(S,\phi)=0$, and, hence, $\xi_{S,\phi}$ is not strongly symplectically fillable \cite{osz2}.
\end{proposition}

In contrast, if $\phi$ is composed solely of right-handed Dehn twists, then $\kh(S,\phi)$ is supported in non-negative homological gradings and $\khc(S,\phi) \neq 0$. Recall that if $L$ is a transverse link in $S^3$, then the element $\khc(L) \in \kh(L)$ is contained in the bi-grading $(0,sl(L)+1)$ \cite{pla1}.
\begin{corollary}
\label{cor:vanishing} If $L$ is a transverse link for which $\kh(L)$ is supported in non-positive homological gradings and $\kh(L)$ vanishes in the bi-grading $(0,sl(L)+1)$, then $\xi_L$ is not strongly symplectically fillable.
\end{corollary}

In Subsection \ref{ssec:nonfill}, we give examples which demonstrate the use of Proposition \ref{prop:vanishing} and Corollary \ref{cor:vanishing} in proving that certain contact structures are not strongly symplectically fillable.

We conclude with an investigation of some additional structure on the reduced Khovanov homology of an open book. Specifically, we show that $\khc(S,\phi)$ behaves naturally with respect to a comultiplication map on $\kh(S,\phi)$. This closely parallels the behavior of the Ozsv{\'a}th-Szab{\'o} contact invariant under a similar comultiplication defined on Heegaard Floer homology in \cite{bald3}.





\subsection*{Acknowledgements} The authors thank Josh Greene, Lenny Ng, and Andras Stipsicz for 
very helpful conversations, and Lawrence Roberts for helpful correspondence.    

\newpage 

\section{The reduced Khovanov homology of an open book}
\label{sec:kh}
Let $S$ be a compact, oriented surface with boundary (we will use the notation $S_{k,r}$ when we wish to emphasize that $S$ has genus $k$ and $r$ boundary components). Recall that the 3-manifold $M_{S,\phi}$ is defined to be $S \times [0,1] /\sim$, where $\sim$ is the identification given by 
\begin{eqnarray*}
(x,1) \sim(\phi(x),0), && x \in S\\
(x,t) \sim (x, s), && x\in \partial S, \ t,s \in [0,1].
\end{eqnarray*} 
Suppose that $\phi$ is a composition of Dehn twists, $\phi=D_{\gamma_1}^{\epsilon_1}\cdots D_{\gamma_n}^{\epsilon_n}$ with $\epsilon_j \in \{-1,1\}$ (writing monodromy as composition $\phi=hg$, we assume that $h$ is performed first).  Choose points $0=t_1<\dots<t_n=1$ in the interval $[0,1]$. Then $-M_{S,\phi}$ is obtained from $-M_{S, id}$ by performing $\epsilon_j$-surgery on $\gamma_j \times \{t_j\}$, relative to the framing induced by $S$, for each $j=1,\dots, n$. See Figure \ref{fig:IdKirby} for an example.

\begin{figure}[!htbp]
\begin{center}
\includegraphics[width=11.5cm]{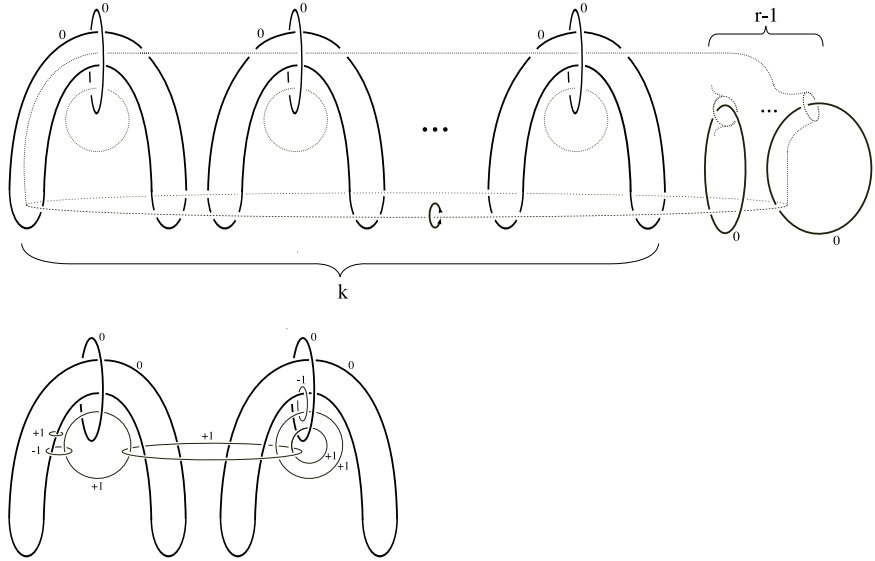}
\caption{\quad At the top is a surgery diagram for $-M_{S_{k,r},id}$. The dotted curves trace out a fiber $S_{k,r} \times \{t\}$. The boundary of this fiber is the binding $B$ of the open book $(S_{k,r},id)$. The arrow indicates the direction of increasing $t$ values (that is, the $S^1$ direction of the fibration). At the bottom is a surgery diagram for $-M_{S_{2,1},D_{\alpha_1}D_{\alpha_2}D_{\alpha_4}D_{\alpha_0}^{-1}D_{\alpha_4}D_{\alpha_1}^{-1}D_{\alpha_3}}$.}
\label{fig:IdKirby}
\end{center}
\end{figure}

For each vector $i=(i_1,\dots, i_n) \in \{0,1\}^n$ we define the ``complete resolution" $(S,\phi)_i$ to be the 3-manifold obtained from $-M_{S, id}$ by performing $m(\epsilon_j,i_j)$-surgery on $\gamma_j \times \{t_j\}$ for each $j=1\dots n$, where $$m(\epsilon_j,i_j)=
\left\{\begin{array}{ll}
0, &\text{if $(\epsilon_j,i_j)=(1,1)$ or $(-1,0)$},\\
\infty, &\text{if $(\epsilon_j,i_j)=(1,0)$ or $(-1,1).$}
\end{array}\right.$$ Said differently, in taking a complete resolution, we replace each $\epsilon_j$-surgery in $-M_{S,\phi}$ with a $0$- or $\infty$-surgery as prescribed above (see Figure \ref{fig:Res} for an example). This is analogous to replacing each crossing in a planar link diagram with one of its two resolutions. We make this analogy more precise in Remark \ref{rmk:braid}.

\begin{figure}[!htbp]
\begin{center}
\includegraphics[width=12.6cm]{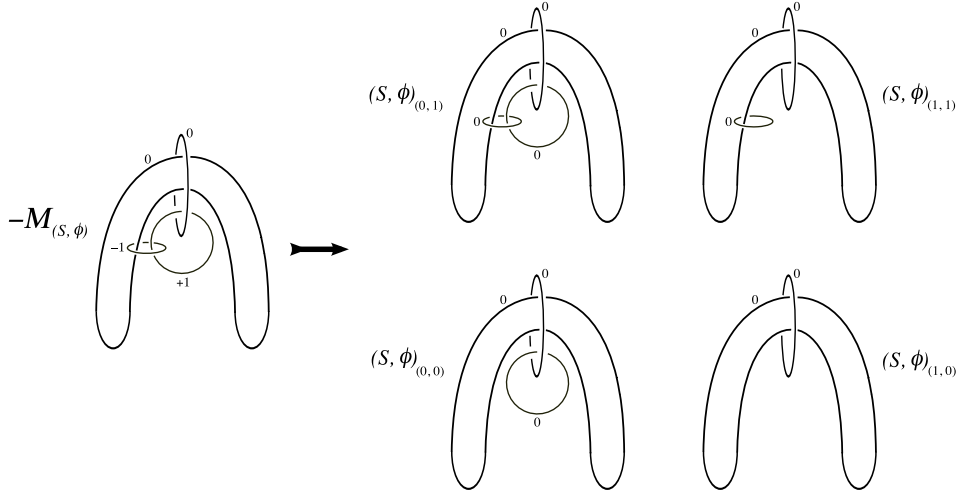}
\caption{\quad Surgery diagrams for the complete resolutions of the open book $(S,\phi)=(S_{1,1},D_{\alpha_2}D_{\alpha_1}^{-1})$.}
\label{fig:Res}
\end{center}
\end{figure}

Following Ozsv{\'a}th and Szab{\'o}, we construct a Heegaard multi-diagram compatible with all possible combinations of $\epsilon_j$-, $0$- and $\infty$-surgeries on the components of the link $L=\bigcup_{j=1}^n \gamma_j \times \{t_j\}$, and we use this to build a chain complex $(C(S,\phi), D)$ whose differential counts holomorphic polygons in a symmetric product of this multi-diagram \cite{osz12}. As a vector space, $$C(S,\phi)=\bigoplus_{i\in \{0,1\}^n}\cf((S,\phi)_i), $$ and $D$ is the sum of maps $$D_{i,i'}:\cf((S,\phi)_i)\rightarrow \cf((S,\phi)_{i'}),$$ over all pairs $i,i'$ for which $i \leq i'$ (we say that $i \leq i'$ if $i_j \leq i'_j$ for all $j=1,\dots,n$).

\begin{theorem}[{\rm \cite[Theorem 4.1]{osz12}}]
\label{thm:homology}
The homology $H_*(C(S,\phi),D)$ is isomorphic to $\hf(-M_{S,\phi})$.
\end{theorem}

There is a grading on $C(S,\phi)$ defined, for $x \in \cf(-\Sigma(L_i))$, by $I(x) = |i|-n_-(\phi)$, where $|i|= i_1+\cdots + i_n$, and $n_-(\phi)$ is the number of left-handed Dehn twists in the composition $\phi$. We refer to this as the ``$I$-grading" or the ``homological grading" on $C(S,\phi)$. This grading induces an ``$I$-filtration" of the complex $(C(S,\phi),D),$ which, in turn, gives rise to a spectral sequence. Let $(E^k_I(S,\phi),D^k_I)$ denote the $(E^k,D^k)$ term of this spectral sequence. The differential $D_I^0$ on the associated graded object $E^0_I(S,\phi)$ is the sum of the standard Heegaard Floer boundary maps $$D_{i,i}:\cf((S,\phi)_i) \rightarrow \cf((S,\phi)_i).$$ Therefore, $E^1_I(S,\phi)$ is isomorphic to $$\bigoplus_{i\in \{0,1\}^n}\hf((S,\phi)_i). $$ 

The vector $i' \in \{0,1\}^n$ is said to be an ``immediate successor" of $i$ if $i'_k > i_k$ for some $k$ and $i'_j = i_j$ for all $j \neq k$. If $i'$ is an immediate successor of $i$, then $(S,\phi)_{i'}$ is obtained from $(S,\phi)_i$ by changing the surgery coefficient on the component $\gamma_k \times \{t_k\}$ from $\infty$ to $0$ or from $0$ to $\infty$. In the first case, $$(D_{i,i'})_*:\hf((S,\phi)_i)\rightarrow \hf((S,\phi)_{i'})$$ is the map induced by the 2-handle cobordism corresponding to $0$-surgery on $\gamma_k \times \{t_k\}$. In the second case, $(D_{i,i'})_*$ is the map induced by the 2-handle cobordism corresponding to $0$-surgery on a meridian of $\gamma_k \times \{t_k\}$.
By construction, the differential $D_I^1$ on $\bigoplus_{i\in \{0,1\}^n}\hf((S,\phi)_i)$ is the sum of the maps $(D_{i,i'})_*$, over all pairs $i,i'$ for which $i'$ is an immediate successor of $i$.

\begin{definition}
\label{def:kh}
The \emph{reduced Khovanov homology} of the open book $(S,\phi)$ is defined to be the graded vector space $E^2_I(S,\phi)$; it is denoted by $\kh(S,\phi)$. 
\end{definition}

The reduced Khovanov homology of a link is a special case of our construction.

\begin{remark}
\label{rmk:braid}
Suppose that $w = \sigma_{l_1}^{\epsilon_1} \cdots \sigma_{l_n}^{\epsilon_n}$ is a word in the generators of the braid group on $m$ strands, with $\epsilon_j \in \{-1,1\}$. For $i \in \{0,1\}^n$, let $(L_w)_i$ denote the link obtained from $L_w$ by taking the $i_j$-resolution (see Figure \ref{fig:Res3}) of the crossing corresponding to $\sigma_{l_j}^{\epsilon_j}$ for each $j=1,\dots, n$. If $m=2k+1$, then $\Sigma(L_w)$ has an open book decomposition given by $(S_{k,1},\phi_w)$, where $\phi_w = D_{\alpha_{l_1}}^{\epsilon_1} \cdots D_{\alpha_{l_n}}^{\epsilon_n}$. By design, the complete resolution $(S_{k,1},\phi_w)_i$ is diffeomorphic to $-\Sigma((L_w)_i)$, and our entire construction of $\kh(S_{k,1},\phi_w)$ is identical to the construction of $\kh(L_w)$ by Ozsv{\'a}th and Szab{\'o} in \cite{osz12}. Moreover, the homological grading on $\kh(S_{k,1},\phi_w)$ agrees with the homological grading on $\kh(L_w)$. 

The same is true for $m=2k+2$. The only difference in this case is that $\Sigma(L_w)$ has an open book decomposition given by $(S_{k,2},\phi_w)$, where $\phi_w = D_{\alpha_{l_1}}^{\epsilon_1} \cdots D_{\alpha_{l_n}}^{\epsilon_n}$ is a composition of Dehn twists around the curves $\alpha_1,\dots, \alpha_{2k+1}$ depicted in Figure \ref{fig:sk2}. 
\end{remark}

\begin{figure}[!htbp]
\begin{center}
\includegraphics[width=7.4cm]{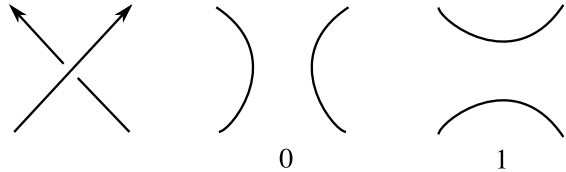}
\caption{\quad Above are the $0$- and $1$-resolutions of a positive crossing. For a negative crossing, the convention is reversed.}
\label{fig:Res3}
\end{center}
\end{figure}

\vfill \eject

\begin{figure}[!htbp]
\begin{center}
\includegraphics[width=7cm]{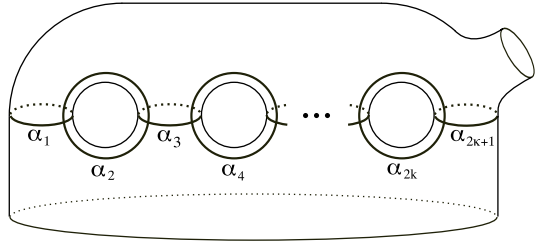}
\caption{\quad The curves $\alpha_ 1,\dots, \alpha_{2k+1}$ on the surface $S_{k,2}$.}
\label{fig:sk2}
\end{center}
\end{figure}

\newpage

\section{Invariance under stabilization}
\label{sec:invc}
Recall that a positive (resp. negative) stabilization of the open book $(S,\phi)$ is an open book $(S',\phi')$, where $S'$ is the union of $S$ with a $1$-handle, and $\phi'$ is the composition of $\phi$ with a right-handed (resp. left-handed) Dehn twist around a curve in $S'$ which intersects the co-core of the $1$-handle exactly once. 

\begin{theorem}
\label{thm:stab}
If $(S',\phi')$ is a stabilization of $(S,\phi)$, then $\kh(S,\phi) \cong \kh(S',\phi')$ as graded vector spaces. 
\end{theorem}

For the proof of Theorem \ref{thm:stab}, we need the following technical result from \cite{osz12}. Recall that for any 3-manifold $Y$, $\hf(Y)$ is a module over the algebra $\wedge^*H_1(Y)/\text{Tors}$ \cite{osz8}. When $Y = \#^l(S^1 \times S^2)$, this module structure is particularly simple. In the following proposition,  $Y_0(K)$ denotes 
the result of the 0-surgery on a knot $K \subset Y$.

\begin{proposition}[\rm{\cite[Proposition 6.1]{osz12}}] 
\label{prop:s1s2}
If $Y = \#^l(S^1 \times S^2)$, then $\hf(Y)$ can be identified with $\wedge^*H_1(Y)$ as an $\wedge^*H_1(Y)$-module. If the knot $K$ represents a circle in one of the $S^1 \times S^2$ summands of $Y$, then $Y' = Y_0(K)$ is diffeomorphic to $\#^{l-1}(S^1 \times S^2)$, and there is a natural identification $$
 \pi: H_1(Y)/[K] \rightarrow H_1(Y'). $$ If $W$ is the corresponding 2-handle cobordism from $Y$ to $Y'$, then the induced map $$F_W:\hf(Y) \rightarrow \hf(Y')$$ is specified by $F_W(\xi) = \pi(\xi).$ Dually, if $K\subset Y$ is an unknot, then $Y'' = Y_0(K)$ is diffeomorphic to $\#^{l+1}(S^1 \times S^2)$, and there is a natural inclusion $$
i: H_1(Y) \rightarrow H_1(Y'').$$ If $W'$ is the corresponding 2-handle cobordism from $Y$ to $Y''$, then the induced map $$F_{W'}:\hf(Y) \rightarrow \hf(Y'')$$ is specified by $F_{W'}(\xi) = i(\xi)\wedge[K''] $ where $[K''] \in H_1(Y'')$ generates the kernel of the map $H_1(Y'') \rightarrow H_1(W').$
\end{proposition}

Our proof of Theorem \ref{thm:stab} is similar in spirit to the proof that the reduced Khovanov homology of a link is invariant under the first Reidemeister move \cite{kh1}. 

\begin{proof}[Proof of Theorem \ref{thm:stab}] 
Let $\phi =D_{\gamma_1}^{\epsilon_1}\cdots D_{\gamma_n}^{\epsilon_n}$. Suppose that $(S',\phi')$ is obtained from $(S,\phi)$ via stabilization, so that $\phi'=\phi \circ \tau_c$, where $\tau_c$ is a Dehn twist around a curve $c$ which intersects the co-core of the new 1-handle exactly once. Let $$L = \bigcup_{j=1}^n \gamma_j \times \{t_j\}\, \cup \, c\times \{t_{n+1}\}$$ be the framed link in $-M_{S',id}$ corresponding to the open book $(S',\phi').$ Write $$L=L'\,\cup \,c\times \{t_{n+1}\}.$$ There is a decomposition \begin{equation*}\label{eqn:decomp} -M_{S',id}\cong -M_{S,id}\, \# \,(S^1\times S^2)\end{equation*} in which the sublink $L'$ is contained in the $-M_{S,id}$ summand, while the component $c\times \{t_{n+1}\}$ represents a circle in the $S^1\times S^2$ summand ($c \times \{t_{n+1}\}$ does not link $L'$). See Figure \ref{fig:stab} for an example.

\begin{figure}[!htbp]
\begin{center}
\includegraphics[width=9cm]{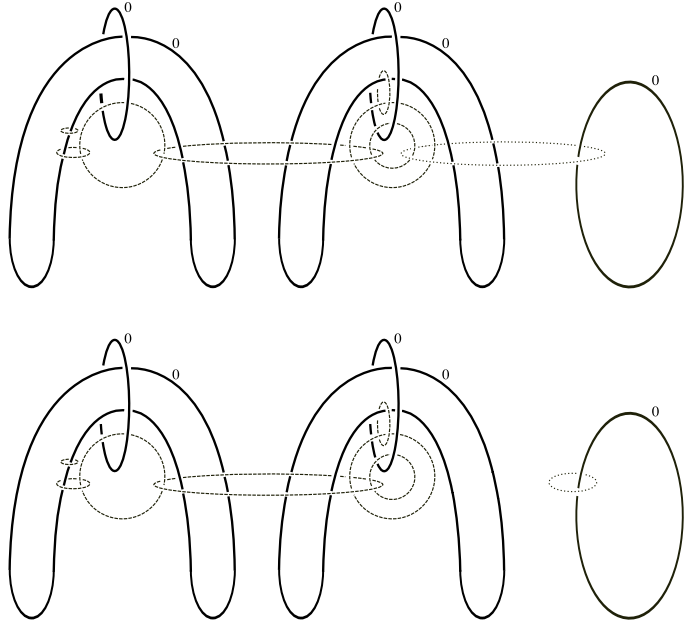}
\caption{\quad The open book in Figure \ref{fig:IdKirby} has a positive stabilization given by $(S_{2,2},D_{\alpha_1}D_{\alpha_2}D_{\alpha_4}D_{\alpha_0}^{-1}D_{\alpha_4}D_{\alpha_1}^{-1}D_{\alpha_3}\circ D_c)$, where $c$ is the curve $\alpha_5\subset S_{2,2}$. The non-solid curves in this diagram represent the link $L$ in $-M_{S_{2,2},id}$ corresponding to this stabilization. The component $c\times \{t_{n+1}\}$ is dotted and the sublink $L'$ is dashed. The bottom figure is obtained from the top figure via handleslides, and represents the decomposition of $-M_{S_{2,2},id}$ into $-M_{S_{2,1},id} \,\#\,(S^1 \times S^2)$. When ``sliding" a link component over a $0$-surgery curve, we are simply taking the connected sum of the component with a ($0$-framed) longitude of the surgery curve. This longitude is an unknot, so the framed link is unchanged by the slide. }
\label{fig:stab}
\end{center}
\end{figure}

Suppose that $(S',\phi')$ is a negative stabilization of $(S,\phi)$. Let $Y_0 \cong S^3$ and $Y_{1} \cong S^1\times S^2$ be the 3-manifolds obtained by performing $0$- and $\infty$-surgeries, respectively, on the component $c \times \{t_{n+1}\} \subset S^1 \times S^2$. Note that $Y_1$ is obtained from $Y_0$ by performing $0$-surgery on a meridian $m_c$ of $c \times \{t_{n+1}\}$. Let $$C_-= \hf(Y_0)\oplus\hf(Y_1)$$ be the complex whose differential $$d:\hf(Y_0) \rightarrow \hf(Y_1)$$ is the map induced by the 2-handle cobordism corresponding to $0$-surgery on $m_c$.

Observe that $L' \subset -M_{S,id}$ is the link associated to the open book $(S,\phi)$. For $i \in \{0,1\}^{n+1}$, let $\underline{i}$ be the vector in $\{0,1\}^{n}$ defined by $\underline i = (i_1,\dots,i_n).$ Then, $$(S',\phi')_i \cong (S,\phi)_{\underline i}\, \# \,Y_{i_n}$$ for all $i \in \{0,1\}^{n+1}$, and the complex $E^1_I(S',\phi')$ decomposes as a tensor product, $$E^1_I(S',\phi') \cong E^1_I(S,\phi) \otimes_{\zzt} C_-.$$ Therefore, \begin{equation}\label{eqn:khneg}\kh(S',\phi') \cong \kh(S,\phi) \otimes_{\zzt} H_*(C_-,d)\end{equation} as a vector space. Using the identification in Proposition \ref{prop:s1s2}, $$C_-\cong \wedge^*H_1(S^3) \oplus \wedge^*H_1(S^1 \times S^2),$$ and the differential $$d:\wedge^*H_1(S^3) \rightarrow \wedge^*H_1(S^1 \times S^2)$$ sends $1$ to a generator of $ \wedge^1H_1(S^1 \times S^2)$, since $m_c$ is an unknot in $Y_0$. Therefore, $H_*(C_-,d)$ is generated by $\wedge^0H_1(S^1 \times S^2) \cong \zzt$. By Equation \ref{eqn:khneg}, $\kh(S',\phi') \cong \kh(S,\phi)$ as vector spaces, and it is clear that this isomorphism preserves the grading.

Now suppose that $(S',\phi')$ is a positive stabilization of $(S,\phi).$ In this case, $Y_0 \cong S^1\times S^1$ and $Y_{1} \cong S^3$ are the 3-manifolds obtained via $\infty$- and $0$-surgeries on $c \times \{t_{n+1}\} \subset S^1 \times S^2$. Let $$C_+= \hf(Y_0)\oplus\hf(Y_1)$$ be the complex whose differential $$d:\hf(Y_0) \rightarrow \hf(Y_1)$$ is the map induced by the 2-handle cobordism corresponding to $0$-surgery on $c \times \{t_{n+1}\}$. As before, $$(S',\phi')_i \cong (S,\phi)_{\underline i}\, \# \,Y_{i_n}$$ for all $i \in \{0,1\}^{n+1}$, and the complex $E^1_I(S',\phi')$ decomposes as a tensor product, $$E^1_I(S',\phi') \cong E^1_I(S,\phi) \otimes_{\zzt} C_+.$$ Using the identification in Proposition \ref{prop:s1s2}, $$C_+\cong \wedge^*H_1(S^1 \times S^2) \oplus\wedge^*H_1(S^3),$$ and the differential $$d:\wedge^*H_1(S^1 \times S^2) \rightarrow \wedge^*H_1(S^3)$$ sends $1$ to $1$, and kills $\wedge^1H_1(S^1 \times S^2)$, since $c \times \{t_{n+1}\}$ is a circle in $Y_0 \cong S^1 \times S^2$. Therefore, $H_*(C_+,d)$ is generated by $\wedge^1H_1(S^1 \times S^2) \cong \zzt$, and $\kh(S',\phi') \cong \kh(S,\phi)$ as in the previous case. Let $\xi$ be a generator of $\wedge^1H_1(S^1 \times S^2)$. The chain map $$\rho:E^1_I(S,\phi) \rightarrow E^1_I(S,\phi)\otimes_{\zzt} C_+$$ defined by sending $x$ to $x\otimes \xi$ induces this isomorphism since $\kh(S',\phi')$ is generated by $E^1_I(S,\phi) \otimes \wedge^1H_1(S^1 \times S^2)$.

\end{proof}

\newpage

\section{The element $\khc(S,\phi)$ and its relationship with $c(S,\phi)$.}
\label{sec:khc}

Suppose that $\phi =D_{\gamma_1}^{\epsilon_1}\cdots D_{\gamma_n}^{\epsilon_n}$, and recall that $$E^1_I(S,\phi) \cong \bigoplus_{i \in \{0,1\}^n} \hf((S,\phi)_i).$$ Let $i_o$ be the vector in $\{0,1\}^n$ defined by $$(i_o)_j=
\left\{\begin{array}{ll}
0, &\text{if $\epsilon_j=1$},\\
1, &\text{if $\epsilon_j=-1.$}
\end{array}\right.$$
Then $(S,\phi)_{i_o}$ is the complete resolution obtained by performing $\infty$-surgery on each curve $\gamma_j \times \{t_j\}$ in $ -M_{S,id}.$ If $S=S_{k,r}$ then $(S,\phi)_{i_o} \cong \#^{2k+r-1} (S^1\times S^2).$ Therefore, we may identify $\hf((S,\phi)_{i_o})$ with $ \wedge^*H_1((S,\phi)_{i_o}),$ by Proposition \ref{prop:s1s2}. Note that $\wedge^{2k+r-1}H_1((S,\phi)_{i_o}) \cong \zzt.$

\begin{definition} $\khct(S,\phi)\in E^1_I(S,\phi)$ is defined to be the generator of $\wedge^{2k+r-1}H_1((S,\phi)_{i_o}).$
\end{definition}

It is not hard to show directly that $\khct(S,\phi)$ is closed in $(E^1_I(S,\phi), D^1_I)$. It is useful, however, to take a slightly more roundabout approach. 

\begin{lemma}
\label{lem:nat} If $(S',\phi')$ is a positive stabilization of $(S,\phi)$, then the chain map $$\rho:E^1_I(S,\phi) \rightarrow E^1_I(S',\phi')$$ defined in the proof of Theorem \ref{thm:stab} sends $\khct(S,\phi)$ to $\khct(S',\phi')$. Recall that $\rho$ induces an isomorphism from $\kh(S,\phi)$ to $\kh(S',\phi')$. 
\end{lemma}

In particular, if $(S',\phi')$ is a positive stabilization of $(S,\phi)$, then $\khct(S,\phi)$ is closed if and only if $\khct(S',\phi')$ is closed since the map $\rho$ is injective.

\begin{proof}[Proof of Lemma \ref{lem:nat}]
For $i\in \{0,1\}^n$, let $\overline i$ be the vector in $\{0,1\}^{n+1}$ defined by $\overline i = (i_1,\dots,i_n,0)$. Recall that $$C_+  \cong \wedge^*H_1(S^1 \times S^2) \oplus\wedge^*H_1(S^3).$$ Restricted to the summand $\hf((S,\phi)_{{i_o}}) ,$  $$\rho: \hf((S,\phi)_{i_o}) \rightarrow\hf((S,\phi)_{i_o})\otimes_{\zzt} C_+$$ sends $\khct(S,\phi)$ to $\khct(S,\phi) \otimes \xi$, where $\xi$ is a generator of $\wedge^1H_1(S^1\times S^2)$. And, with respect to the identification $$\hf((S,\phi)_{i_o})\otimes_{\zzt} \wedge^*H_1(S^1\times S^2)\cong \hf((S',\phi')_{\bar{i_o}}) ,$$ $\khct(S,\phi) \otimes \xi$ corresponds to $\khct(S',\phi')$. 

\end{proof}

For the rest of this section, we will assume that $S$ has connected boundary unless otherwise specified. Roberts' main observation in \cite{lrob1}, transplanted to our more general setting, is that the binding $B$ of the open book $(S,\phi)$ gives rise to an additional filtration of the complex $(C(S,\phi),D)$. More precisely, $B$ may be incorporated into the Heegaard multi-diagram mentioned in Section \ref{sec:kh} so that the intersection points in the multi-diagram which generate the summands $\cf((S,\phi)_i)$ of $C(S,\phi)$ are each assigned an Alexander grading (as $B$ gives rise to a null-homologous knot $B_i$ in each $(S,\phi)_i$). These Alexander gradings (or ``$A$-gradings" for short) induce a filtration of the group $C(S,\phi)$, which we call the ``$A$-filtration". The differential $D$ is a filtered map with respect to the $A$-filtration since none the curves $\gamma_j \times \{t_j\}$ algebraically link $B$ \cite[Section 8]{osz3}. 

If $C$ is a filtered group with filtration $$\{0\} = \mathcal{F}_{j-l}\subset \ldots \subset\mathcal{F}_{j-1}\subset\mathcal{F}_{j} = C,$$ then we refer to $\mathcal{F}_k$ as ``filtration level $k$". If $d$ is a differential on $C$ which respects this filtration, then there is an induced filtration of homology, $$\{0\} = \mathcal{F}'_{j-l}\subset \ldots \subset\mathcal{F}'_{j-1}\subset\mathcal{F}'_{j} = H_*(C,d),$$ where $\mathcal{F}'_k$ is the set of homology classes which can be represented by cycles in $\mathcal{F}_k$. 


The $A$-filtration of $(C(S,\phi),D)$ gives rise to an obvious filtration of $E^0_I(S,\phi)$, which, in turn, induces an $A$-filtration of each term $E^n_I(S,\phi)$. 

\begin{lemma}
\label{lem:psi}
$\khct(S,\phi)$ is the unique non-zero element of $E^1_I(S,\phi)$ in filtration level $-g(S)$. Moreover, $\khct(S,\phi)$ is closed in $(E^1_I(S,\phi), D^1_I)$.
\end{lemma}

Every open book can be transformed into an open book with connected binding after a sequence of positive stabilizations. Therefore, Lemmas \ref{lem:nat} and \ref{lem:psi} imply that $\khct(S,\phi)$ is closed for any $S$.

\begin{proof}[Proof of Lemma \ref{lem:psi}] Let $S = S_{k,1}$. Then, $\khct(S,\phi)$ is the generator of $\wedge^{2k}H_1((S,\phi)_{i_o}).$ The natural isomorphism between $\hf((S,\phi)_{i_o})$ and $\hfk((S,\phi)_{i_o},B_{i_o})$ gives an identification of $\hfk((S,\phi)_{i_o},B_{i_o},j)$ with the group $\wedge^{k-j}H_1((S,\phi)_{i_o})$ \cite{osz3}. As a result, $\khct(S,\phi)$ corresponds to the generator of $\hfk((S,\phi)_{i_o},B_{i_o},-k),$ and is therefore a non-zero element of $E^1_I(S,\phi)$ in filtration level $-k=-g(S).$ 

If $i \neq i_o$ then $(S,\phi)_i$ is obtained from $M_{S,id}$ by performing $0$-surgery on at least one of the curves $\gamma_j \times \{t_j\}$. Therefore, $g(B_i)<g(S)$, and, hence, $\hfk((S,\phi)_i,B_i,j)$ vanishes for $j \leq -g(S).$ As a result, there is no non-zero element of $\hf((S,\phi)_i)$ in $A$-filtration level $-g(S)$. Since the differential $D^1_I$ is a filtered map with respect to the $A$-filtration of $E^1_I(S,\phi),$ it follows that $D^1_I(\khct(S,\phi)) = 0$.

\end{proof}

\begin{definition}
\label{def:khc} For any open book $(S,\phi)$, we define $\khc(S,\phi)$ to be the image of $\widetilde{\khc}(S,\phi)$ in $E^2_I(S,\phi)=\kh(S,\phi)$.
\end{definition}

Note that $\khc(S,\phi)$ is contained in homological grading $0$. Below, we show that $\khc$ vanishes for negative stabilizations.

\begin{lemma}\label{lem:negstab}
Let $(S,\phi)$ be any open book, and suppose that $(S',\phi')$ is a negative stabilization of $(S,\phi)$. Then $\khc(S',\phi')=0$.\end{lemma}

\begin{proof}[Proof of Lemma \ref{lem:negstab}]
As in the proof of Theorem \ref{thm:stab}, $$E^1_I(S',\phi') \cong E^1_I(S,\phi) \otimes_{\zzt} C_-,$$ where $C_-$ is the complex $$C_-\cong \wedge^*H_1(S^3) \oplus \wedge^*H_1(S^1 \times S^2),$$ whose differential $$d:\wedge^*H_1(S^3) \rightarrow \wedge^*H_1(S^1 \times S^2)$$ sends $1$ to a generator $\xi$ of $ \wedge^1H_1(S^1 \times S^2).$ Under this identification, the summand $\hf((S',\phi')_{i_o})$ is identified with $$\hf((S,\phi)_{\underline{i_o}})\otimes_{\zzt} \wedge^*H_1(S^1\times S^2),$$ and $\khct(S',\phi')$ corresponds to $\khct(S,\phi)\otimes \xi$. Since $\khct(S,\phi)$ is closed in $E^1_I(S,\phi)$,  $\khct(S,\phi)\otimes \xi$ is the boundary of $\khct(S,\phi)\otimes 1$ in this complex. Therefore, $\khc(S',\phi')=0$.

\end{proof}

\begin{remark}
By Theorem \ref{thm:stab} and Lemma \ref{lem:nat}, it is enough to prove Theorem \ref{thm:tight} for open books with connected binding.
\end{remark}

\begin{remark}It is not hard to see directly that the invariant of the transverse $(2k+1)$-braid $L_w$ defined by the second author in \cite{pla1} is precisely the class $\khc(S_{k,1},\phi_w)$ in $ \kh(S_k,\phi_w)\cong \kh(L_w).$ This also follows from Lemma \ref{lem:psi} combined with Roberts' characterization of $\khc(L_w)$ in \cite{lrob1}.
\end{remark}



Let us denote by $E^n_A(S,\phi)$ the $E^n$ term of the spectral sequence associated to the $A$-filtration of $(C(S,\phi),D)$. In \cite{lrob1}, Roberts (effectively) proves that, for $n \geq 1$, $E^n_A(S,\phi)$ is isomorphic as a graded group to the $E^n$ term of the spectral sequence associated to the Alexander filtration of $\cf(-M_{S,\phi})$ induced by $B$ \cite[Lemma 7]{lrob1}. We may therefore view the contact invariant $c(S,\phi)$ as the unique generator of $E^{\infty}_A(S,\phi)$ in $A$-grading $-g(S)$ \cite{osz1}. Equivalently, $c(S,\phi)$ is the generator of the $A=-g(S)$ filtration level of $H_*(C(S,\phi),D)$.

In order to make the relationship between $\khc(S,\phi)$ and $c(S,\phi)$ more transparent, we provide a short review of the ``cancellation lemma," and describe how it is used to compute spectral sequences.

\begin{lemma}[\rm{see \cite[Lemma 5.1]{ras}}] 
\label{lem:cancel}
Suppose that $(C,d)$ is a complex over $\zzt$, freely generated by elements $x_i,$ and let $d(x_i,x_j)$ be the coefficient of $x_j$ in $d(x_i)$. If $d(x_k,x_l)=1,$ then the complex $(C',d')$ with generators $\{x_i| i \neq k,l\}$ and differential $$d'(x_i) = d(x_i) + d(x_i,x_l)d(x_k)$$ is chain homotopy equivalent to $(C,d)$. The chain homotopy equivalence is induced by the projection $\pi:C\rightarrow C'$, while the equivalence $\iota: C'\rightarrow C$ is given by $\iota(x_i) = x_i + d(x_i,x_l)x_k$.
\end{lemma}

 We say that $(C',d')$ is obtained from $(C,d)$ by ``canceling" the component of the differential $d$ from $x_k$ to $x_l$. Lemma \ref{lem:cancel} admits a refinement for filtered complexes. In particular, suppose that there is a grading on $C$ which induces a filtration of the complex $(C,d)$, and let the elements $x_i$ be homogeneous generators of $C$. If $d(x_k,x_l) = 1$, and $x_k$ and $x_l$ have the same grading, then the complex obtained by canceling the component of $d$ from $x_k$ to $x_l$ is \emph{filtered} chain homotopy equivalent to $(C,d)$ since both $\pi$ and $\iota$ are filtered maps in this case. 
 
Computing the spectral sequence associated to such a filtration is the process of performing cancellation in a series of stages until we arrive at a complex in which the differential is zero (the $E^{\infty}$ term). The $E^n$ term records the result of this cancellation after the $n$th stage. Specifically, the $E^0$ term is simply the graded vector space $C = \bigoplus C_i$. The $E^1$ term is the graded vector space $C^{(1)}$, where $(C^{(1)},d^{(1)})$ is obtained from $(C,d)$ by canceling the components of $d$ which do not shift the grading. For $n>1$, the $E^n$ term is the graded vector space $C^{(n)}$, where $(C^{(n)},d^{(n)})$ is obtained from $(C^{(n-1)},d^{(n-1)})$ by canceling the components of $d^{(n-1)}$ which shift the grading by $n-1$. See Figure \ref{fig:complex} for an illustration of this process (in this diagram, the generators are represented by dots and the components of the differential are represented by arrows).

\begin{figure}[!htbp]
\begin{center}
\includegraphics[width=12.7cm]{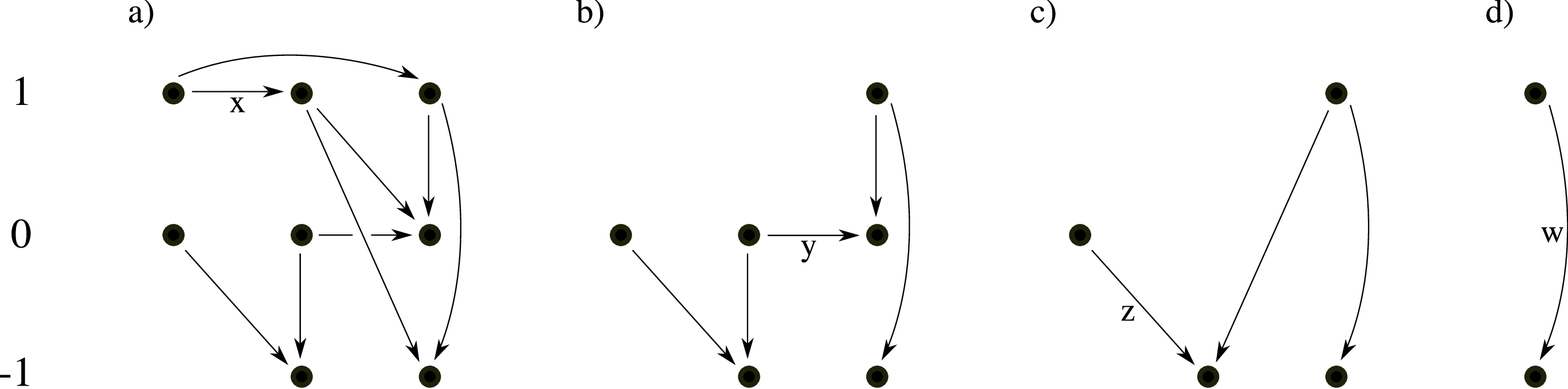}
\caption{\quad The diagram in $a)$ represents a graded complex $C$, where the grading of a generator is given by $1$, $0$, or $-1$. This grading induces a filtration $\mathcal{F}_{-1} \subset \mathcal{F}_0 \subset \mathcal{F}_1 = C$. The complex in $b)$ is obtained from that in $a)$ by canceling the component $x$ of the differential. The complex in $c)$ is obtained from that in $b)$ by canceling $y$. This graded vector space represents the $E^1$ term of the spectral sequence associated to the filtration of $C$. The complex in $d)$ is obtained from that in $c)$ by canceling $z$, and it represents the $E^2$ term of the spectral sequence. The $E^3 = E^{\infty}$ term of the spectral sequence is trivial, and is obtained from the complex in $d)$ by canceling $w$.}
\label{fig:complex}
\end{center}
\end{figure}

Let us now apply the cancellation lemma to the complex $(C(S,\phi),D)$. As mentioned above, $(C(S,\phi),D)$ is generated by the intersection points of the Heegaard multi-diagram from Section \ref{sec:kh}. These generators are homogeneous with respect to both the $I$-grading and the $A$-grading on $C(S,\phi)$. Canceling all components of $D$ between generators in the same $(I,A)$-bi-grading, we obtain a complex $(C(S,\phi)',D')$ which is bi-filtered chain homotopy equivalent to $(C(S,\phi),D)$ (since $\pi$ and $\iota$ are both bi-filtered maps). Let us denote by $E_I^n(S,\phi)'$ and $E^n_A(S,\phi)'$ the $E^n$ terms of the spectral sequences associated to the $I$- and $A$-filtrations of $(C(S,\phi)',D')$ (these terms are isomorphic as graded groups to $E^n_I(S,\phi)$ and $E^n_A(S,\phi)$).

As a bi-filtered group, $C(S,\phi)'$ is clearly isomorphic to $$\bigoplus_{i\in\{0,1\}^n} \hfk((S,\phi)_i,B_i).$$ Consequently, the generator $c$ of $\hfk((S,\phi)_{i_o}, B_{i_o},-g(S))$ is the unique non-zero element of $C(S,\phi)'$ in $A$-filtration level $-g(S)$ (see the proof of Lemma \ref{lem:psi}). Thus, according to Definition \ref{def:khc}, the element $\khc(S,\phi)$ may be viewed as the image of $c$ in $E^2_I(S,\phi)'\cong \kh(S,\phi).$ Similarly, the contact invariant $c(S,\phi)$ may be viewed as the image of $c$ in $H_*(C(S,\phi)',D') \cong \hf (-M_{S,\phi}).$ \emph{This} is the true sense in which $\khc(S,\phi)$ ``corresponds" to $c(S,\phi)$.

In this context, Theorem \ref{thm:tight} boils down to the statement that if the spectral sequence associated to the $I$-filtration of $(C(S,\phi)',D')$ collapses at $E^2_I(S,\phi)'$ and the image of $c$ in $E^2_I(S,\phi)'$ is non-zero, then the image of $c$ in $H_*(C(S,\phi)',D')$ is non-zero. 

\begin{proof}[Proof of Theorem \ref{thm:tight}]
Let us assume that the spectral sequence associated to the $I$-filtration of $(C(S,\phi)',D')$ collapses at $E^2_I(S,\phi)'$, and suppose that $c$ is exact in $(C(S,\phi)',D')$. Cancel all components of $D'$ between generators in the same $I$-grading to obtain a new complex $(C(S,\phi)'',D'')$, where $C(S,\phi)'' = E^1_I(S,\phi)'$. Note that there is no component of $D'$ between two generators of $\hfk((S,\phi)_{i_o}, B_{i_o})$ since this group is isomorphic to $\hf((S,\phi)_{i_o})$. Therefore, $c$ represents an exact element of $(C(S,\phi)'',D'')$. 

Now, cancel all components of $D''$ from generators $x_k$ to $x_l$, where $x_l \neq c$ and $I(x_l) = I(x_k)+1$. The element $c$ remains exact in the resulting complex $(C(S,\phi)''',D''')$, and any component of $D'''$ which shifts the $I$-grading by $1$ must map to $c$. If such a component exists, then the image of $c$ is zero in the group $E^2_I(S,\phi)'$ (which is obtained from $C(S,\phi)'''$ by canceling this component), and we are done. If no such component exists, then $C(S,\phi)''' = E^2_I(S,\phi)'$, which implies that $D'''$ is trivial since we are assuming that the spectral sequence collapses at $E^2_I(S,\phi)'$. But this contradicts the fact that $c$ is exact in $(C(S,\phi)''',D''')$. 

\end{proof}

\begin{remark}
\label{rmk:corr} One should not read too much into the ``correspondence" between $\khc(S,\phi)$ and $c(S,\phi)$. For instance, it is possible that $\khc(S,\phi) = 0$ while $c(S,\phi) \neq 0$, perhaps even when the spectral sequence associated to the $I$-filtration of $(C(S,\phi)',D')$ collapses at $E^2_I(S,\phi)'$. An example of this sort of phenomenon is given by the model complex $(C,d)$ generated by the three elements $x_{0,0}$, $y_{1,-2}$, $z_{2,-1}$, with $dx=y+z$ (here, the subscripts indicate the $(I,A)$-bi-grading). The element $y$ in the lowest $A$-grading plays the role of $c$. The $E^1$ term of the spectral sequence associated to the $I$-filtration of $(C,d)$ is generated by the elements $x$, $y$, $z$ as well, and $d^1(x) = y$ (think $\khc(S,\phi)=0$). Yet, $y$ is not a boundary in $(C,d)$ (think $c(S,\phi) \neq 0$). Moreover, this spectral sequence collapses at the $E^2$ term.
\end{remark}


The hypotheses in Proposition \ref{prop:vanishing} are designed to avoid the type of phenomenon described in Remark \ref{rmk:corr}. Applied to the complex $(C(S,\phi)',D')$, Proposition \ref{prop:vanishing} says that if the image of $c$ in $E^2_I(S,\phi)'$ is zero and $E^2_I(S,\phi)'$ is supported in non-positive $I$-gradings, then the image of $c$ in $H_*(C(S,\phi)',D')$ is zero.

\begin{proof}[Proof of Proposition \ref{prop:vanishing}]
Suppose that $\phi =D_{\gamma_1}^{\epsilon_1}\cdots D_{\gamma_n}^{\epsilon_n}$, and let $K = n-n_-(\phi)$. If $\mathcal{F}_j$ is the subset of $C(S,\phi)'$ generated by homogeneous elements with $I$-gradings $\geq j$, then $$\{0\} = \mathcal{F}_{K+1} \subset  \mathcal{F}_{K} \subset \dots \subset  \mathcal{F}_{-n_-(\phi)} = C(S,\phi)'.$$
Let us assume that $E^2_I(S,\phi)'$ is supported in non-positive $I$-gradings. If the image of $c$ in $E^2_I(S,\phi)'$ is zero, then there must exist some $y$ with $I(y)=-1$ such that $D'(y) = c + x$, where $x \in \mathcal{F}_1.$ Let $k$ be the greatest integer for which there exists some $y'$ such that $D'(y') = c + x'$, where $x' \in \mathcal{F}_k$. We will show that $k=K+1$, which implies that $x'=0$, and, hence, that $c$ is a boundary in $(C(S,\phi)',D')$.

Suppose, for a contradiction, that $k<K+1$. Write $x' = x_k + x''$, where $I(x_k) = k$, and $x'' \in \mathcal{F}_{k+1}$. Note that $D'(x_k+x'')=0$ as $x' = x_k + x''$ is homologous to $c$. Since every component of $D'$ shifts the $I$-grading by at least $1$, it follows that $D'(x'') \in \mathcal{F}_{k+2}.$ But this implies that $D'(x_k) \in  \mathcal{F}_{k+2}$ as well, since $D'(x_k + x'') = 0$. Therefore, $x_k$ represents a cycle in $(E^1_I(S,\phi)',(D'_I)^1)$. Since $k\geq 1$ and $E^2_I(S,\phi)'$ is supported in non-positive $I$-gradings, it must be that $x_k$ is also a boundary in $(E^1_I(S,\phi)',(D'_I)^1)$. That is, there is some $y''$ with $I(y'') = k-1$ such that $D'(y'') = x_k + x'''$, where $x''' \in \mathcal{F}_{k+1}$. But then, $D'(y'+y'') = c+(x'' + x''')$, and the fact that $x'' + x'''$ is contained in $\mathcal{F}_{k+1}$ contradicts our earlier assumption on the maximality of $k$.

\end{proof}
\newpage 

\section{The computability of $\kh(S,\phi)$ and $\khc(S,\phi)$}
\label{sec:compute}
As was mentioned in the introduction, the mapping class group of $S_{k,1}$ is generated by Dehn twists around the curves $\alpha_0, \dots, \alpha_{2k}$ depicted in Figure \ref{fig:genset}. In this section, we show that $\kh(S_{k,1},\phi)$ and $\khc(S_{k,1},\phi)$ are combinatorially computable when $\phi$ is a composition of Dehn twists around the curves shown in Figure \ref{fig:genset3}, which include $\alpha_0, \dots, \alpha_{2k}$. 

According to Proposition \ref{prop:s1s2}, if $Y = \#^l (S^1 \times S^2)$, then the map on Heegaard Floer homology induced by a 2-handle cobordism corresponding to $0$-surgery on an unknot in $Y$ or on a circle in an $S^1 \times S^2$ summand of $Y$ depends only on homological data. The combinatoriality of $\kh(S_{k,1},\phi)$ and $\khc(S_{k,1},\phi)$ therefore follows directly from the lemma below.

\begin{figure}[!htbp]
\begin{center}
\includegraphics[width=7cm]{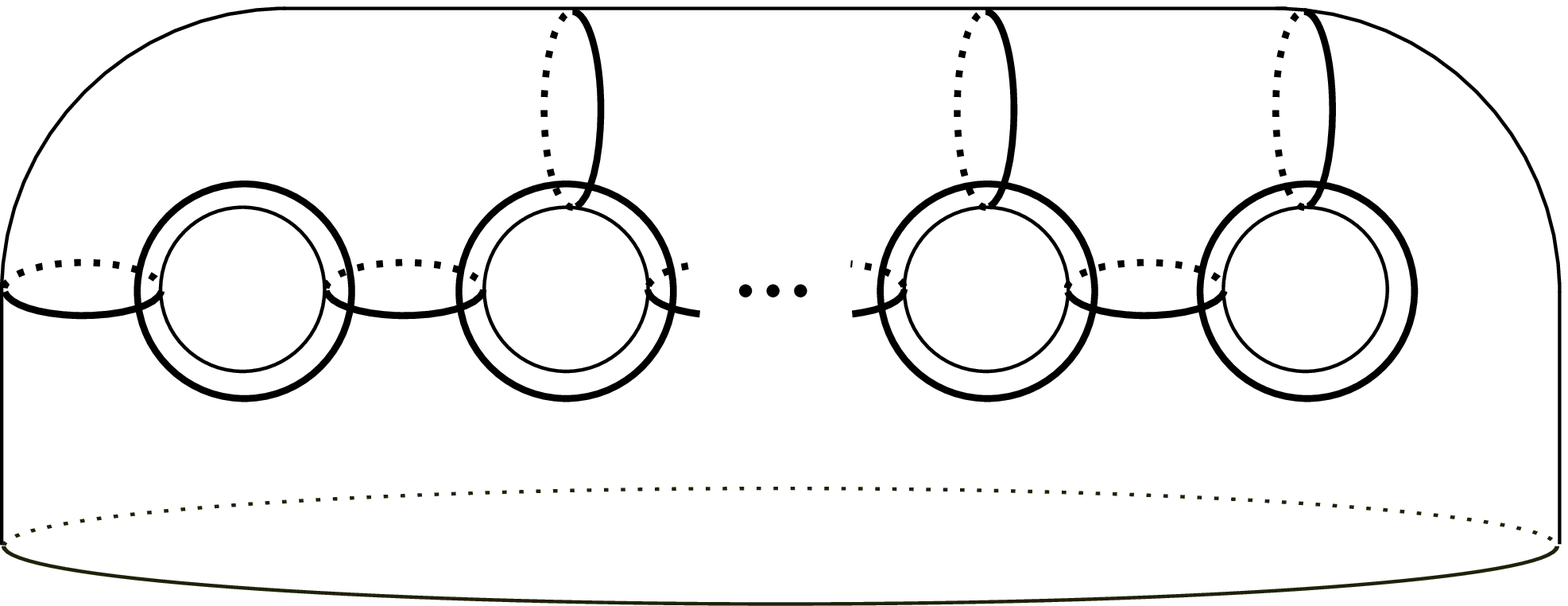}
\caption{}
\label{fig:genset3}
\end{center}
\end{figure}

\begin{lemma}
\label{lem:ress1s2}
Suppose that $\phi$ is a composition of Dehn twists around the curves depicted in Figure \ref{fig:genset3}. Then, each $(S_{k,1},\phi)_i$ is diffeomorphic to $\#^l(S^1 \times S^2)$ for some $l$. Moreover, if $i'$ is an immediate successor of $i$, then the map $$(D_{i,i'})_*:\hf((S_{k,1},\phi)_i)\rightarrow \hf((S_{k,1},\phi)_{i'})$$ is induced by a 2-handle cobordism corresponding to $0$-surgery on either an unknot in $(S_{k,1},\phi)_i$ or on a circle in one of the $S^1\times S^2$ summands of $(S_{k,1},\phi)_i$.
\end{lemma}

\begin{proof}[Proof of Lemma \ref{lem:ress1s2}]
Let $\phi$ be the composition $\phi=D_{\gamma_1}^{\epsilon_1}\cdots D_{\gamma_n}^{\epsilon_n},$ where each $\gamma_j$ is among the curves depicted in Figure \ref{fig:genset3}. As described in Section \ref{sec:kh}, a surgery diagram for the complete resolution $(S_{k,1},\phi)_i$ is obtained from the diagram for $-M_{S_{k,1},id}$ (depicted in Figure \ref{fig:IdKirby}) by performing $0$-surgeries on the curves in some subset of $\{\gamma_j \times \{t_j\}\}_{j=1}^n$. 

After isotopy, such a diagram consists of several ``blocks" of concentric circles, together with ``staggered" horizontal and vertical circles (see Figure \ref{fig:block}). Consider the slightly more general arrangement of curves represented schematically in Figure \ref{fig:block2}. In this schematic picture, the shaded annuli represent blocks of concentric circles, the unmarked rectangles represent staggered horizontal and vertical circles, and the rectangle labeled $A$ represents a union of horizontal circles which may be slid past one another like beads on an abacus, up and down the strands of the concentric circles in the rightmost block. Certainly, the curves in the surgery diagram for $(S_{k,1},\phi)_i$ constitute such an arrangement.

\begin{figure}[!htbp]
\begin{center}
\includegraphics[width=10.5cm]{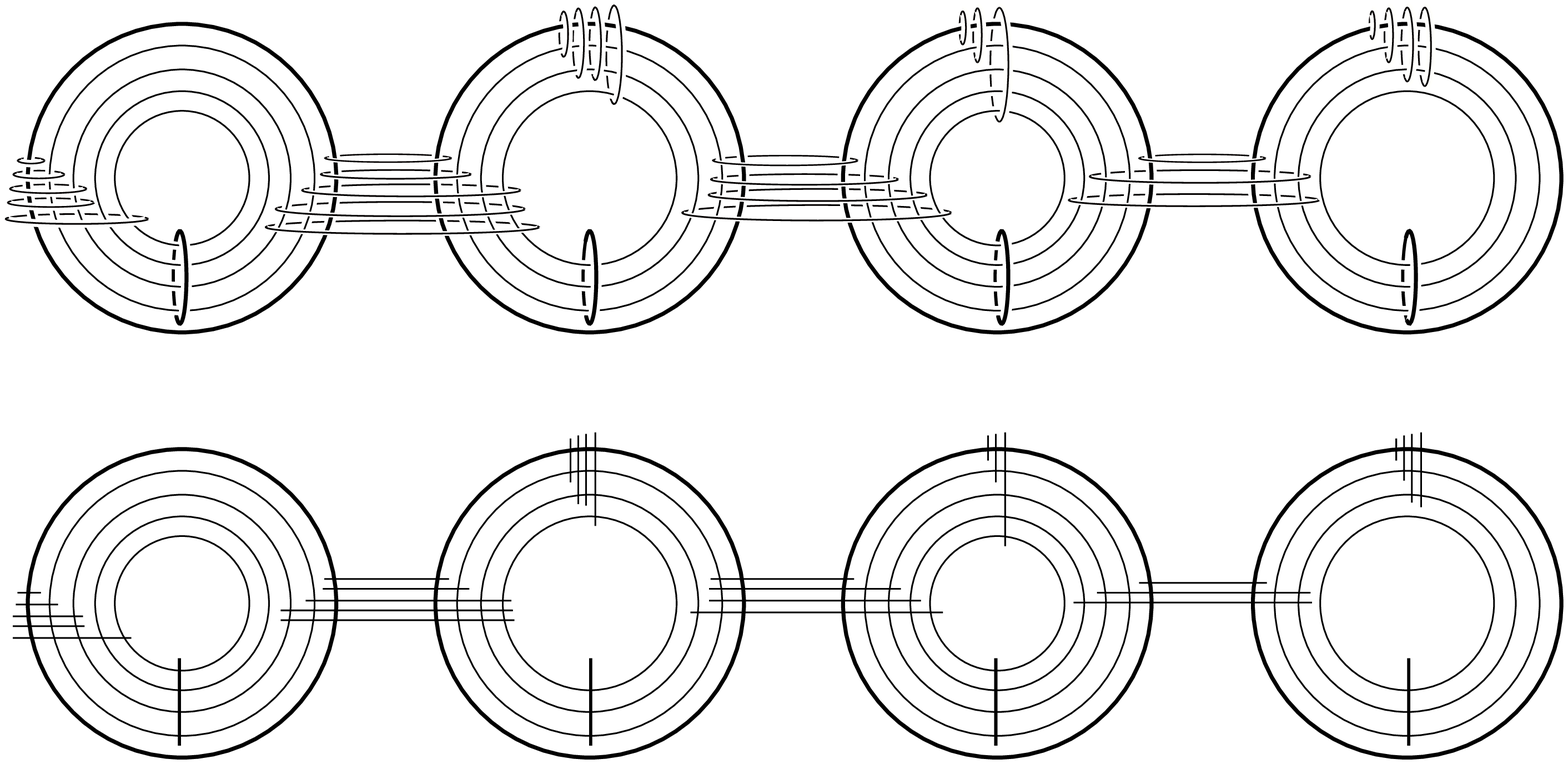}
\caption{At the top is an example of the surgery diagram for a complete resolution $(S_{4,1},\phi)_i$. Though we have not labeled the surgery coefficients, it is understood that they are all $0$. The thicker circles are the surgery curves for $-M_{S_{4,1},id}$. The bottom picture is shorthand for the surgery diagram at the top.}
\label{fig:block}
\end{center}
\end{figure}

\begin{figure}[!htbp]
\begin{center}
\includegraphics[width=13.5cm]{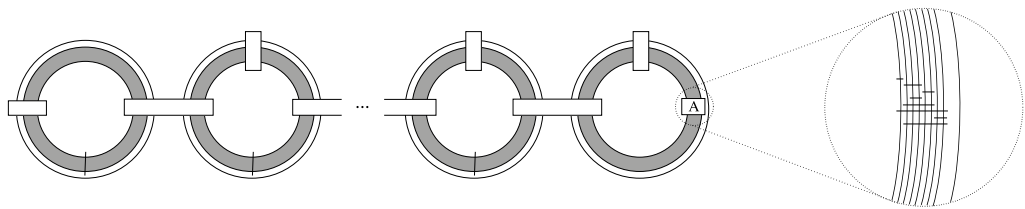}
\caption{\quad The magnified region is an example of what the horizontal circles represented by the rectangle $A$ might look like. The key point is that you can slide any one of them past any other.}
\label{fig:block2}
\end{center}
\end{figure}

Our proof that $(S_{k,1},\phi)_i$ is a connected sum of $S^1 \times S^2$'s is rather inelegant. We start with a diagram obtained by performing $0$-surgery on each of the curves in an arrangement as in Figure \ref{fig:block2}, and perform handleslides and handle cancellations until we arrive at a diagram for $0$-surgeries on the components of an unlink. 

We perform these handleslides/cancellations beginning with the rightmost block of concentric circles. If $C$ is the innermost of these concentric circles, then $C$ might link some of the horizontal and vertical circles represented by the rectangles which intersect the rightmost annulus in Figure \ref{fig:block2}. If there is no such linking, then $C$ is an unknot which may be pulled aside. If there \emph{is} linking, then handleslide these horizontal and vertical circles over one another until only one remains which links $C$ (it does not matter which circle is handleslid over which), and then cancel $C$ with this remaining surgery curve (see Figure \ref{fig:block3} for an illustration of this process). In either case, we are left with one fewer concentric circle in our arrangement. It is not hard to see that this procedure may be repeated until all of the concentric circles in the rightmost block have been cancelled or pulled aside from the arrangement. Afterwards, what remains is an arrangement of surgery curves as in Figure \ref{fig:block2} with one fewer block of concentric circles, together with an unlink. Repeat this process until all of the blocks have been eliminated, and all that is left is an unlink.

\begin{figure}[!htbp]
\begin{center}
\includegraphics[width=10cm]{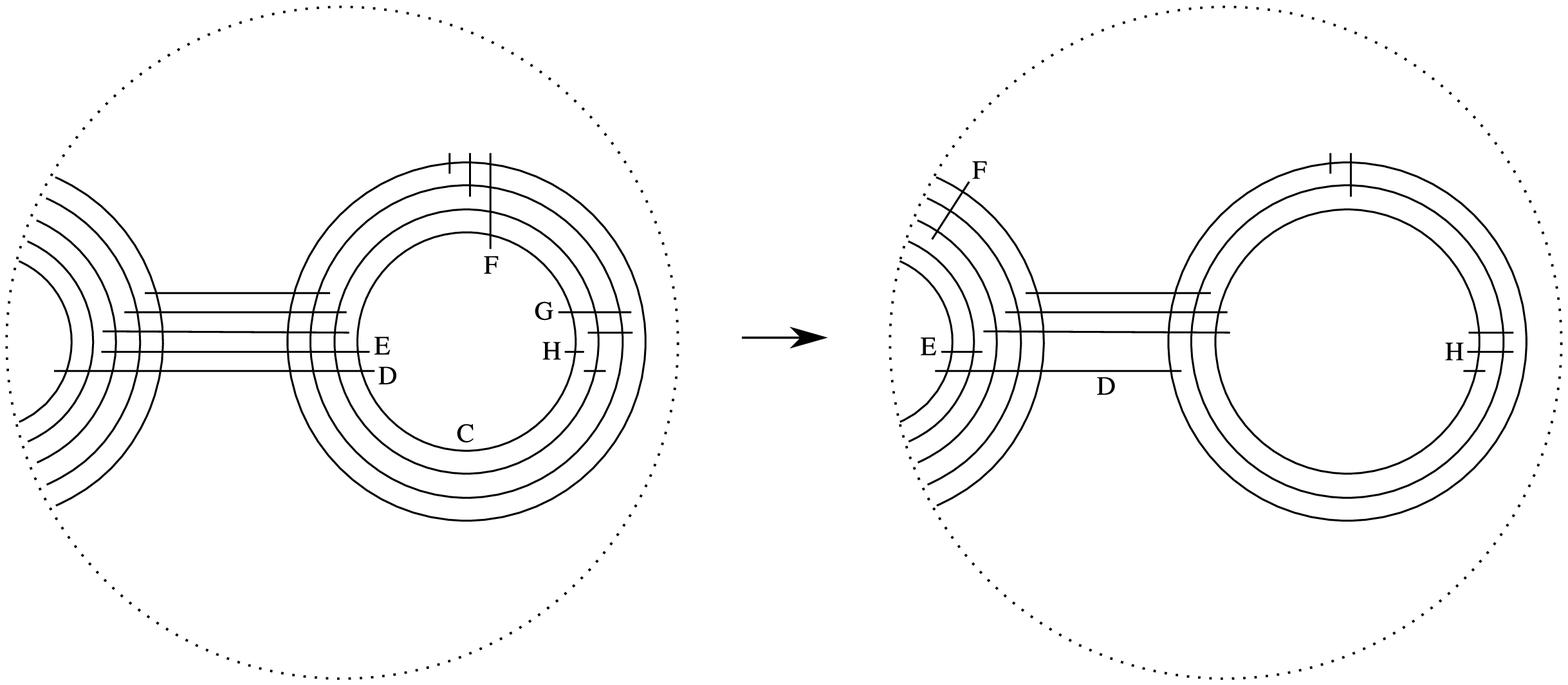}
\caption{\quad In order, handleslide $F$ over $E$, then $E$ over $D$, then $D$ over $G$, and, finally, $H$ over $D$. Next, cancel $C$ with $G$. The result is the picture on the right. We can iterate this process until we have eliminated all of the concentric circles in the rightmost block.}
\label{fig:block3}
\end{center}
\end{figure}

As was described in Section \ref{sec:kh}, the maps $$(D_{i,i'})_*:\hf((S,\phi)_i)\rightarrow \hf((S,\phi)_{i'})$$ are induced by 2-handle cobordisms corresponding either to $0$-surgery on $\gamma_k \times \{t_k\}$ for some $k$, or on a meridian of $\gamma_k \times \{t_k\}$. Therefore, to finish the proof of Lemma \ref{lem:ress1s2}, it suffices to show that if $K$ is one of the curves in the arrangement in Figure \ref{fig:block2}, then $K$, thought of as a knot in the 3-manifold $Y$ obtained by performing $0$-surgery on all of the \emph{other} curves in the arrangement, is either an unknot (with the correct framing) or a circle in one of the $S^1 \times S^2$ summands of $Y$. But this can be seen explicitly by keeping track of the knot $K$ as one performs the reductive algorithm described above. 

\end{proof}

In order to actually compute $\kh(S_{k,1},\phi)$, where $\phi$ is a composition of Dehn twists around the curves in Figure \ref{fig:genset3}, begin by fixing a basis for the first homology of each complete resolution. By Proposition \ref{prop:s1s2} and Lemma \ref{lem:ress1s2}, $$\hf((S_{k,1},\phi)_i) \cong \wedge^*H_1((S_{k,1},\phi)_i).$$ Now suppose that $i'$ is an immediate successor of $i$ for which $i'_k>i_k$, and determine how the basis for $H_1((S_{k,1},\phi)_i)$ changes upon performing $0$-surgery on either $\gamma_k \times \{t_k\}$ or its meridian. With this information in hand, we may use (according to Lemma \ref{lem:ress1s2}) Proposition \ref{prop:s1s2} to compute the map $$(D_{i,i'})_*:\wedge^*H_1((S,\phi)_i)\rightarrow \wedge^*H_1((S,\phi)_{i'})$$ in terms of the bases that we fixed at the beginning. This is how the program $\tt{Kh}$ \cite{baldurl} works. Below, we give a simple example of this procedure.

\begin{example} On the left of Figure \ref{fig:Res2} are surgery diagrams for the complete resolutions of the open book $(S,\phi) = (S_{1,1},D_{\alpha_2} D_{\alpha_1}^{-1})$. We have omitted the surgery coefficients, but it is understood that they are all $0$ (compare with Figure \ref{fig:Res}). Instead, we have labeled the surgery curves $A$, $B$, $C$, and $D$. Let $a$, $b$, $c$, and $d$ denote the respective meridians of these curves. The right side of Figure \ref{fig:Res2} depicts the term $E^1_I(S,\phi) \cong \bigoplus_{i\in \{0,1\}^2}\hf((S,\phi)_i)$ using the identification of $\hf((S,\phi)_i)$ with $\wedge^*H_1((S,\phi)_i)$ described in Proposition \ref{prop:s1s2}. In addition, we have indicated the maps $(D_{i,i'})_*$ which comprise the differential $D^1_I$.

\begin{figure}[!htbp]
\begin{center}
\includegraphics[width=13.6cm]{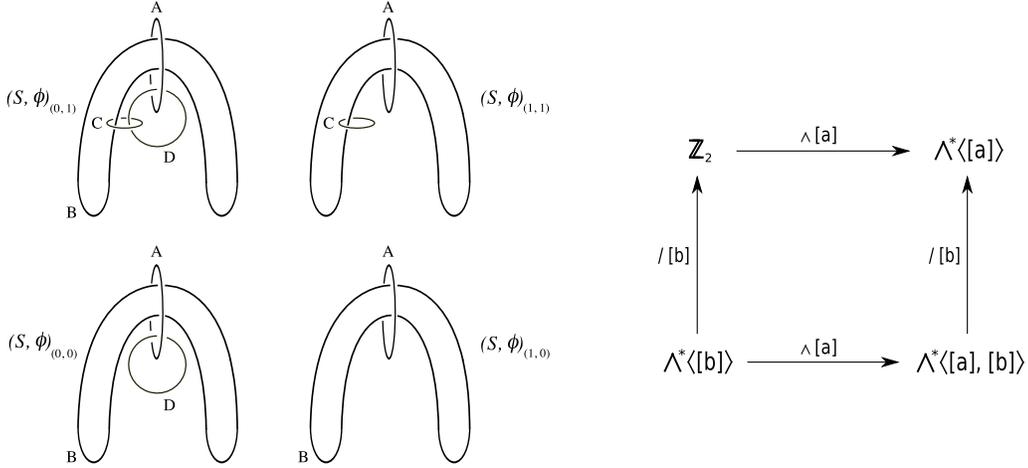}
\caption{\quad Computing the reduced Khovanov homology of $(S,\phi) = (S_{1,1},D_{\alpha_2} D_{\alpha_1}^{-1})$.}
\label{fig:Res2}
\end{center}
\end{figure}

Observe that $(S,\phi)_{(0,0)} \cong S^1 \times S^2$, and its first homology is generated by $[b]$. Therefore, $\hf((S,\phi)_{(0,0)})$ may be identified $\wedge^*\langle [b] \rangle$.  Meanwhile, $(S,\phi)_{(1,0)} \cong \#^2(S^1 \times S^2)$, and its first homology is generated by $[a]$ and $[b]$; hence, $\hf((S,\phi)_{(0,0)})$ may be identified with $\wedge^*\langle[a], [b] \rangle$. Now, $(S,\phi)_{(1,0)}$ is obtained from $(S,\phi)_{(0,0)}$ by performing $0$-surgery on the meridian $d.$ Handlesliding $d$ over $A$, we see that $d$ is an unknot in $(S,\phi)_{(0,0)}$. If $W$ is the 2-handle cobordism corresponding to this $0$-surgery, then the kernel of the map $H_1((S,\phi)_{(1,0)}) \rightarrow H_1(W)$ is generated by $[a]$. Therefore, by Proposition \ref{prop:s1s2}, the map $$(D_{(0,0),(1,0)})_*:\wedge^*\langle [b] \rangle \rightarrow \wedge^*\langle[a], [b] \rangle$$ sends $x$ to $x \wedge [a]$. The other components $(D_{i,i'})_*$ are computed similarly.

It is easy to see that $\kh(S,\phi)$ is isomorphic to $\zzt$, and is generated by the cycle corresponding to $[b] \in \hf((S,\phi)_{(1,0)}).$ Moreover, the element $\khct(S,\phi)=[a]\wedge[b] \in \hf((S,\phi)_{(1,0)})$ is the boundary of $[b] \in \hf((S,\phi)_{(0,0)})$ under the differential $D^1_I$. Therefore, $\khc(S,\phi)=0$. These results are not surprising since $\kh(S,\phi)$ is simply the reduced Khovanov homology of the closed braid $L_{\sigma_2 \sigma_1^{-1}}$, which is the unknot. In addition, $\khc(S,\phi)$ is Plamenevskaya's transverse link invariant, $\khc(L_{\sigma_2 \sigma_1^{-1}})$, which is known to vanish \cite[Proposition 3]{pla1}. (Of course,  vanishing of $\khc(S,\phi)=0$
also follows from Lemma \ref{lem:negstab}, since the open book $(S,\phi)$ can be obtained by negative stabilization.)
\end{example}

Our generalization of reduced Khovanov homology was motivated, in part, by a conjecture of Ozsv{\'a}th which suggests that $\text{rk}(\kh(K_1)) = \text{rk}(\kh(K_2))$ whenever $\Sigma(K_1) \cong \Sigma(K_2)$.\footnote{Liam Watson has recently found an infinite family of counterexamples \cite{liam}.} This conjecture would follow from our construction if $\kh(S,\phi)$ were an invariant of the 3-manifold $M_{S,\phi}$. Although $\kh(S,\phi)$ is invariant under stabilization, $\kh(S,\phi)$ is not an invariant of the isotopy class of $\phi$. This can be shown using the program {\tt Kh}. 

According to \cite{hum} (see also \cite{ozstip}), the mapping class group of $S_{k,1}$ has a presentation $$\langle D_{\alpha_0},\dots,D_{\alpha_{2k}}| b_0,c_1,\dots,c_{2k-1}, r_1, r_2\rangle,$$ where $b_0$ is the relation $$(D_{\alpha_4}D_{\alpha_0}D_{\alpha_4})(D_{\alpha_0}D_{\alpha_4}D_{\alpha_0})^{-1},$$ $c_j$ is the relation $$(D_{\alpha_j}D_{\alpha_{j+1}}D_{\alpha_j})(D_{\alpha_{j+1}}D_{\alpha_j}D_{\alpha_{j+1}})^{-1},$$ and $r_1$ is the relation $\phi_1 \phi_2^{-1},$ where $$\phi_1 = (D_{\alpha_1}D_{\alpha_2}D_{\alpha_3}) ^4D_{\alpha_4}^{-1}D_{\alpha_3}^{-1}D_{\alpha_2}^{-1},$$ and $$\phi_2 = D_{\alpha_0}D_{\alpha_4}^{-1}D_{\alpha_3}^{-1}D_{\alpha_2}^{-1}D_{\alpha_1}^{-1}D_{\alpha_0}D_{\alpha_1}^{-1}D_{\alpha_4}D_{\alpha_2}^{-1}D_{\alpha_3}D_{\alpha_2}D_{\alpha_4}^{-1}D_{\alpha_1}D_{\alpha_0}^{-1}D_{\alpha_1}.$$ The relation $r_2$ is more complicated. Computer calculations suggest that $\kh(S_{k,1},\phi)$ is invariant under composition with the relations $b_0$ and $c_j$, but not with $r_1$, since $\kh(S_{2,1},\phi_1) \cong \zzt^6$, while $\kh(S_{2,1},\phi_2) \cong \zzt^{25}.$

For a while, we had hoped that the \emph{vanishing} of $\khc(S,\phi)$ was an invariant the contact structure $\xi_{S,\phi}$. We now know this to be false. Assume, for a contradiction, that it is true. Then $\khc(S,\phi)=0$ whenever $\xi_{S,\phi}$ is overtwisted. For, if $\xi_{S,\phi}$ is overtwisted, then $\xi_{S,\phi}$ is compatible with an open book $(S,\phi')$ which is a negative stabilization of some other open book (see \cite{ozstip}). By Lemma \ref{lem:negstab}, $\khc(S,\phi')=0$. On the other hand, the open book $(S,\phi) = (S_{2,1}, D_{\alpha_1} D_{\alpha_2} D_{\alpha_3} D_{\alpha_5} (D_{\alpha_3})^{-5} D_{\alpha_4} D_{\alpha_0})$ in Example \ref{nex2} below corresponds to an overtwisted contact structure, while the program {\tt Trans} \cite{baldurl} shows that $\khc(S,\phi) \neq 0.$

\newpage 

\section{Transverse links}
\label{sec:bdc}

In this section, we prove Theorem \ref{thm:nonzero} and Corollary \ref{cor:tightness}. In Section \ref{sec:examples}, we will use these two results in conjunction with Theorem \ref{thm:tight} to show that the contact structure $\xi_K$ is tight for several classes of transverse knots $K\subset S^3$.

\begin{proof}[Proof of Theorem \ref{thm:nonzero}] Suppose that $K$ is a transverse knot with $sl(K)=s(K)-1$. Though we are interested in reduced Khovanov homology with coefficients in $\zzt$, it is instructive to first consider the case of non-reduced homology with rational coefficients. In \cite{lee}, Lee introduces a differential $d'= d + \Phi$ on the Khovanov chain complex $CKh(K)$, where $d$ is Khovanov's original differential \cite{kh1}, and $\Phi$ is a map which raises the quantum grading (or ''$q$-grading" for short). Recall that $d'$ arises from the multiplication and comultiplication maps defined by     
\begin{eqnarray*}
 m(\veep\otimes \veep) = m(\veem\otimes \veem)= \veep \qquad    &\Delta(\veep) = \veep\otimes \veem + \veem\otimes\veep \\ 
 m(\veep\otimes \veem) = m(\veep\otimes \veem)= \veem \qquad    &\Delta(\veem) = \veem\otimes \veem + \veep\otimes\veep 
\end{eqnarray*}
(we follow the notation of \cite{ras3}, in which $\veem$ and $\veep$ are the standard generators in quantum degrees $-1$ and $+1$.) 

Lee proves that $Kh'(K) = H_*(CKh(K),d')$ is isomorphic to $\mathbb{Q} \oplus \mathbb{Q}$, and is generated by two canonical cycles which correspond to the two possible orientations of $K$. Represent the transverse knot by an oriented $k$-braid, and let $\so$ denote the corresponding canonical cycle. The oriented resolution of this braid is a union of $k$ nested circles, and $$\so=(\veem + (-1)^d\veep)\otimes(\veem + (-1)^{d+1}\veep)\otimes \dots \otimes (\veem + (-1)^{d+k-1} \veep)$$ is the element of $CKh(K)$ obtained by alternately labeling these circles by $\veem +\veep$ and $\veem-\veep$ (the label on the outermost circle depends on the orientation of the knot). Recall that $\khc(K)=[\widetilde{\khc}(K)] \in Kh(K)$, where $$\widetilde{\khc}(K) = \veem \otimes\veem\otimes \dots \otimes \veem$$ is the $q$-homogeneous part of $\so$ with the lowest $q$-grading \cite{pla1}.

In \cite{ras3}, Rasmussen defines a function $s$ on $Kh'(K)$ whose value on $x\in Kh'(K)$ is the largest $n$ such that $x$ can be represented by a cycle whose $q$-homogeneous terms all have $q$-gradings at least $n$. The invariant $s(K)$ is then defined so that $s(K)-1=s_{min}=s([\so])$. Recall from \cite{pla1} that $q(\widetilde\khc(K))=sl(K)$. Therefore, the hypothesis of Theorem \ref{thm:nonzero} implies that $q(\widetilde{\khc}(K))=s(K)-1$. If $\khc(K)=0$, then $\widetilde{\khc}(K) = dy$ for some $y \in CKh(K)$. Since $d$ preserves quantum gradings, it must be that $q(y)=s(K)-1$ as well. Then $\so-d'y$ is a cycle in $(CKh(K),d')$ which is homologous to $\so$ and whose $q$-homogeneous terms all have $q$-gradings strictly greater than $s(K)-1$. This contradicts the equality $s([\so])=s(K)-1$. 

To make a similar argument with $\zzt$ coefficients, we must use Turner's modification of Lee's construction (since $Kh'(K)$ is isomorphic to $Kh(K)$ over $\zzt$). In \cite{tur}, Turner defines a differential $d''$ on $CKh(K)$ (now, with $\zzt$ coefficients) using the multiplication and comultiplication maps given by
\begin{eqnarray*}
 m(\veep\otimes \veep) &=& \veep  \qquad \qquad \qquad \qquad \Delta(\veep) = \veep\otimes \veem + \veem\otimes\veep +\veep\otimes\veep \\ 
 m(\veem\otimes \veem)&=& \veem  \qquad  \qquad  \qquad \qquad \Delta(\veem) = \veem\otimes \veem  \\
 m(\veep\otimes \veem) &=& m(\veep\otimes \veem)= \veem.   
\end{eqnarray*}
He shows that $Kh''(K) = H_*(CKh(K),d'')$ is isomorphic to $\zzt \oplus\zzt$, and is generated by two cycles corresponding to the two orientations of $K$. If we represent $K$ by an oriented braid, then the corresponding cycle $\so$ is obtained by alternately labeling the nested components of the oriented resolution by $\veem+\veep$ and $\veem$ (as before, the label on the outermost circle depends on the orientation of the knot). The transverse invariant $\psi(K)$ is again the image in $Kh(K)$ of the $q$-homogeneous part of $\so$ with the lowest $q$-grading. Moreover, an analogue of the function $s$ can be defined in this setting, and it takes the same values as Rasmussen's $s$ function \cite{mvt}. Therefore, our argument from the preceding paragraph still applies.  

It remains to deal with the reduced case. Mark a point on $K$, and let $CKh_{-}(K)$ denote the subcomplex generated by elements in which the marked circle is labeled by $\veem$ in every resolution. The reduced Khovanov chain complex is the quotient complex $$\widetilde{CKh}(K) = CKh(K)/CKh_{-}(K).$$ Turner's construction works just as well for reduced Khovanov homology over $\zzt$. In this case, $\kh''(K) = H_*(\widetilde{CKh}(K),d'')$ is isomorphic to $\zzt$. When $K$ is a braid, $\kh''(K)$ is generated by the element $\so$, which is obtained by alternately labeling the nested components of the oriented resolution of $K$ by $\veem+\veep$ and $\veem$ so that the marked circle is labeled by $\veem+\veep$. The $q$-homogeneous part of $\so$ with the lowest $q$-grading is the element $\widetilde{\psi}(K)$ formed by labeling the marked circle in the oriented resolution of $K$ by $\veep$, and every other circle by $\veem$. The quantum grading is shifted by 1 in the reduced theory, so that the analogue $\widetilde{s}$ of Rasmussen's function $s$ satisfies $\widetilde{s}([\so])=s(K)$. Since $q(\widetilde{\khc}(K)) = sl(K)+1$ in the reduced theory, we may proceed as before to show that $sl(K) = s(K)-1$ implies that $\khc(K) \neq 0$.    

To prove the converse, note that 
if $\kh(K)$ is thin, then the component $\kh^{0, *}(K)$ is non-trivial only in the quantum grading $q=s(K)$. Since $\khc(K) \in \kh^{0, *}(K)$, and $q(\khc(K)) = sl(K)+1$, it follows that if $\khc(K) \neq 0$, then $sl(K)=s(K)-1$. 

\end{proof}

\begin{proof}[Proof of Corollary \ref{cor:tightness}] This follows immediately from Theorem \ref{thm:nonzero} and Theorem \ref{thm:tight}, together with the fact that if $K$ is quasi-alternating, then the spectral sequence from $\kh(K)$ to $\hf(-\Sigma(K))$ collapses at the $E^2$ term and $s(K)=\sigma(K)$ \cite{manozs}. 

\end{proof}

\newpage

\section{Examples}
\label{sec:examples}

\subsection{Tightness by means of Khovanov homology}
\label{ssec:tightness}
In this subsection, we give examples of contact structures whose tightness can be established by means of Khovanov homology. We first focus on the case of double covers of transverse braids for which Corollary 
\ref{cor:tightness} applies. When $K$ is a quasipositive braid, the equality $sl(K) = s(K)-1$ holds, but the contact structure $\xi_K$ is Stein-fillable and, therefore, automatically tight. To get more interesting examples, we look for non-quasipositive knots for which $sl = s-1$. The mirrors of $10_{125}$, $10_{130},$ and $10_{141}$ are the only such knots with 10 crossings or fewer. (As was indicated to the second author by Lenny Ng, this may be verified by contrasting the list of quasipositive knots from \cite{Baa} and the values of the maximal self-linking numbers \cite{ng}.) We use each of these knots to obtain an infinite family of tight contact 
structures, and we show that the contact structures in these families are not Stein-fillable.   

We also use Theorem \ref{thm:tight} to give examples of tight contact 3-manifolds which do not obviously arise as branched double covers of $(S^3,\xi_{st})$ along transverse links. Unfortunately, we do not have any analogues of Theorem \ref{thm:nonzero} and Corollary \ref{cor:tightness} in this more general situation. Instead, we use the computer programs {\tt Kh} and {\tt Trans} \cite{baldurl} to check the collapsing of the spectral sequence and non-vanishing of $\psi$. 
 
It will be helpful to use contact surgery presentations for the contact structures we consider. The idea \cite{AO, pla3} in the construction of such surgery diagrams is to isotope a page of the open book and certain curves on that page so that all curves on which Dehn twists are performed become Legendrian knots. Right-handed Dehn twists are then equivalent to Legendrian surgeries, and left-handed Dehn twists to $(+1)$ contact surgeries. We can (and will) always assume that the monodromy presentation of an open book starts with a sequence of right-handed Dehn twists whose product yields an open book compatible with $(S^3,\xi_{st})$; this allows us to get rid of 1-handles and perform all the contact surgeries on Legendrian knots in $S^3$.  More precisely, a sequence of Dehn twists performed in a certain order corresponds to a sequence of surgeries on push-offs of the corresponding curves. Some care is needed to determine the linking of these push-offs; we refer the reader to \cite{HKP} for details (in the case of branched covers), and only state the answer for the more general case that we need here. The procedure of ``Legendrianizing'' an open book is illustrated on Figure \ref{fig:surgery_diagrams}. A page of an open book for $S^3$ with the monodromy given by $D_{\alpha_2} D_{\alpha_3} \dots D_{\alpha_n}$ can be visualized as the Seifert surface $S$ of a $(2, n)$ torus knot; the curves $\alpha_1$, \dots, $\alpha_n$ become loops around subsequent Hopf bands forming the surface $S$, and $\alpha_0$ is a somewhat more complicated curve (see top of Figure \ref{fig:surgery_diagrams}). Next, the torus knot can be placed into $S^3$ (with the standard contact form $dz-y\,dx$) so that $S$ becomes a page of an open book compatible with the contact structure; the curves $\alpha_0, \dots \alpha_n$ are now Legendrian knots that can be seen on this page  (bottom of Figure \ref{fig:surgery_diagrams}). (Some care must be taken to place two strands of $\alpha_0$ on the same thin strip; this can be achieved by slightly twisting the strip.)

Considering various push-offs of these curves (shown on Figure \ref{fig:pushoffs}, see also \cite{HKP}) step-by-step, we can obtain surgery diagrams corresponding to arbitrary monodromies. We observe that 
the curves  $\alpha_1$, \dots, $\alpha_n$ (and their push-offs) yield Legendrian unknots with $tb=-1$, so 
the diagrams for branched double covers all consist of standard Legendrian unknots only; 
the curve $\alpha_0$ gives a stabilized Legendrian unknot with $tb=-2$.

\begin{figure}[!htbp] 
\includegraphics[width = 12.5cm]{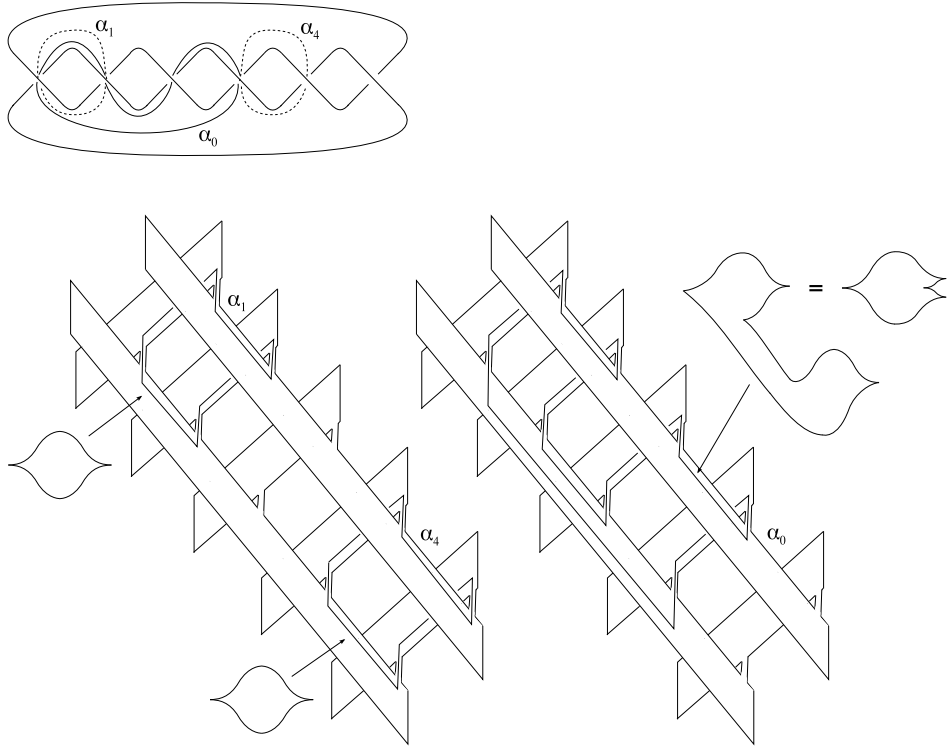}
\caption{From open books to surgery diagrams.}\label{fig:surgery_diagrams} 
\end{figure}  
  
\begin{figure}[!htbp] 
\includegraphics[width = 11cm]{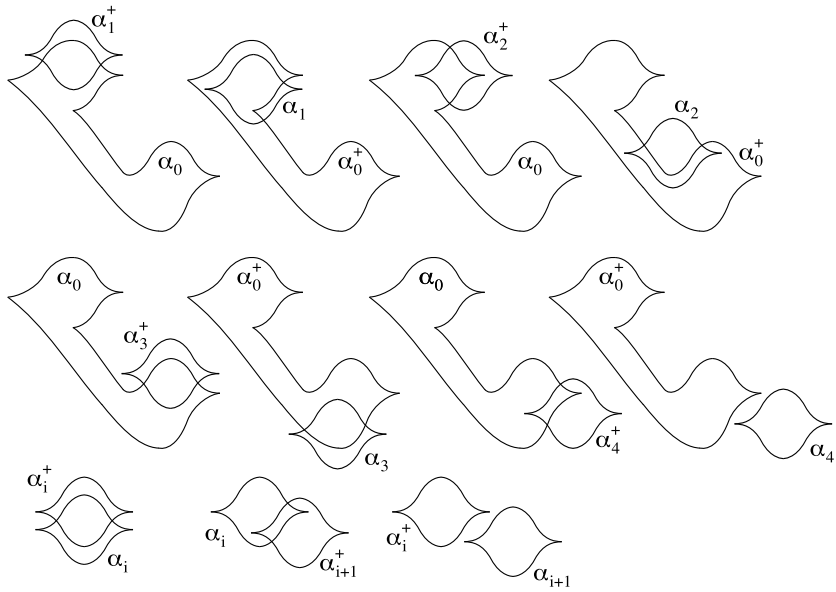}
\caption{The linking of Legendrian push-offs of the curves $\alpha_i$.}\label{fig:pushoffs} 
\end{figure}

\begin{example} \label{ex:e125} 

For $r\geq 5$,  consider the pretzel link $P(-r, 3, -2)$ (for $n=5$, this is the mirror of the knot $10_{125}$), and let $K_r$ be its transverse representative given by the closed braid $$(\sigma_1)^{-r} \sigma_2 \sigma_1^{3} \sigma_2.$$
We use the algorithm described above to obtain the contact surgery description 
for the induced contact structure $\xi$ on the branched double cover  $\Sigma(K_r)$.  The resulting surgery diagram for $r=5$ is  shown on the right of Figure~\ref{125} (when $r>5$, we get the diagram  with $r$ $(+1)$-surgeries). 
The unoriented surgery link is Legendrian isotopic to the one shown on the left of Figure~\ref{125}, so we can 
work with the more symmetric diagram. In other examples of this section, we will similarly change surgery links 
by Legendrian isotopy to make pictures more symmetric.

\begin{figure}[!htbp] 
\includegraphics[width = 11cm]{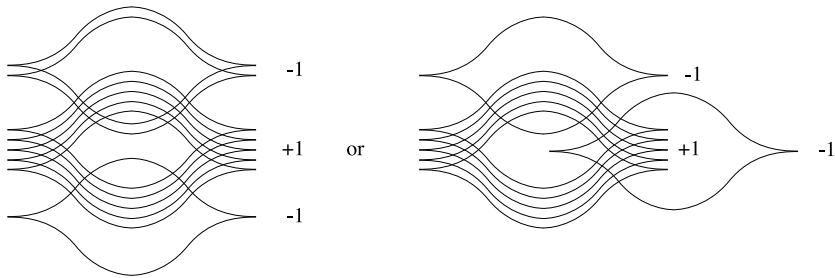}
\caption{The surgery diagram for the branched double cover  of the transverse knot $K=K_5$ in Example \ref{ex:e125}
 .}\label{125} 
\end{figure}

The underlying smooth manifold  $\Sigma(K_r)$ is  the Seifert fibered space $M(-1; 2/3, 1/2, 1/n)$.  (See Figure~\ref{kirby125} for a sequence of Kirby calculus  moves demonstrating this for $n=5$.) Each link $K_r$ is quasi-alternating. Indeed, $|\det (K_r)| = |H_1 ( M(-1; 2/3, 1/2, 1/n) )| = r+6$. On the other hand, resolving the crossing circled in Figure \ref{braid125} in two possible ways, we obtain the link $K_{r-1}$ and the unknot. Repeating the procedure $r$ times, we get the link $K_0$, which is the trefoil linked once with the unknot.
Thus $K_0$ is an alternating link with $|\det (K_0)|=6$, and, since $|\det (K_r)| = |\det(K_{r-1})|  + |\det(\text{unknot})|$, we see by induction that $K_r$ is quasi-alternating.

\begin{figure}[!htbp] 
\includegraphics[width = 9cm]{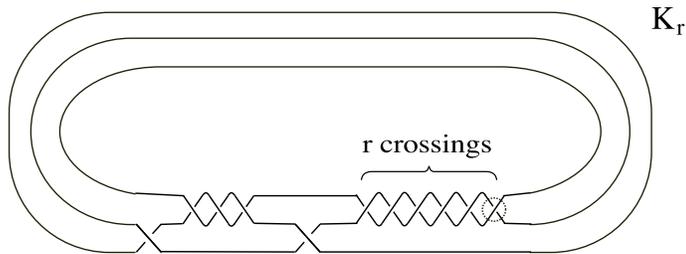}  
\caption{The links $K_r$ are quasi-alternating.}\label{braid125}
\end{figure} 

We next check the hypothesis of Theorem \ref{thm:nonzero} when $r$ is odd (i.e. $K_r$ is a knot). We compute $sl(K)= 2-r$. Since the knot $K_r$ is quasi-alternating, $s$ equals to the signature $\sigma(K_r)=3-r$ (we compute the signature via the Goeritz matrix of the knot \cite{GL}). When $r$ is even, $K_r$ is a two-component link, so Theorem \ref{thm:nonzero} does not apply. However, we can argue that $\psi(K_r) \neq 0$ by \cite[Theorem 4]{pla1}, since $\psi(K_{r-1}) \neq 0$, and the transverse braid $K_{r-1}$ is obtained from  $K_{r}$ by resolving a negative crossing. 

Corollary \ref{cor:tightness}  implies that the branched double cover of each $K_r$ is a tight contact manifold. We now show that none of them are Stein-fillable. Since $\Sigma(K_5)$ can be obtained from any of $\Sigma(K_r)$ by a sequence of Legendrian surgeries, it suffices to check that $\xi=\xi_{K_5}$  (shown on Figure~\ref{125}) is not Stein-fillable.

First, we compute the homotopy invariants of the contact structure $\xi$. Recall \cite{DGS} that the three-dimensional invariant $d_3$ of a contact structure given by a contact surgery diagram can be computed as $$d_3 (\xi) = \frac{c_1(\spc)^2-2\chi(X)-3 \sigma(X)+2}4 +m, $$ where $X$ is a 4-manifold bounded by $Y$ and obtained by adding $2$-handles to $B^4$ as dictated by the surgery diagram, $\spc$ is the corresponding $\Sc$ structure on $X$, and $m$ is the number of $(+1)$-surgeries in the diagram. The $\Sc$ structure $\spc$ arises from an almost-complex structure defined in the complement of a finite set of points in $X$, and the class  $c_1(\spc)$ evaluates on each homology generator of $X$ corresponding to the handle attachment along an (oriented) Legendrian knot as the rotation number of the knot. Thus, for the contact structure  $\xi$  on $Y=M(-1; 2/3, 1/2, 1/5)$  defined by the surgery diagram from Figure~\ref{125}, we compute  $c_1(\spc)=0$ and $d_3(\xi)=-\frac12$.

We show that $\xi$ is not Stein-fillable, combining the ideas from \cite{GLS} and \cite{Li}. More precisely, 
we will show that $Y$ carries no Stein-fillable contact structures with $d_3=-\frac12$. Note that $Y$ is an $L$-space since it is the branched double cover 
of a quasi-alternating knot \cite{osz12}. 
\begin{figure}[!htbp] 
\includegraphics[width = 11.3cm]{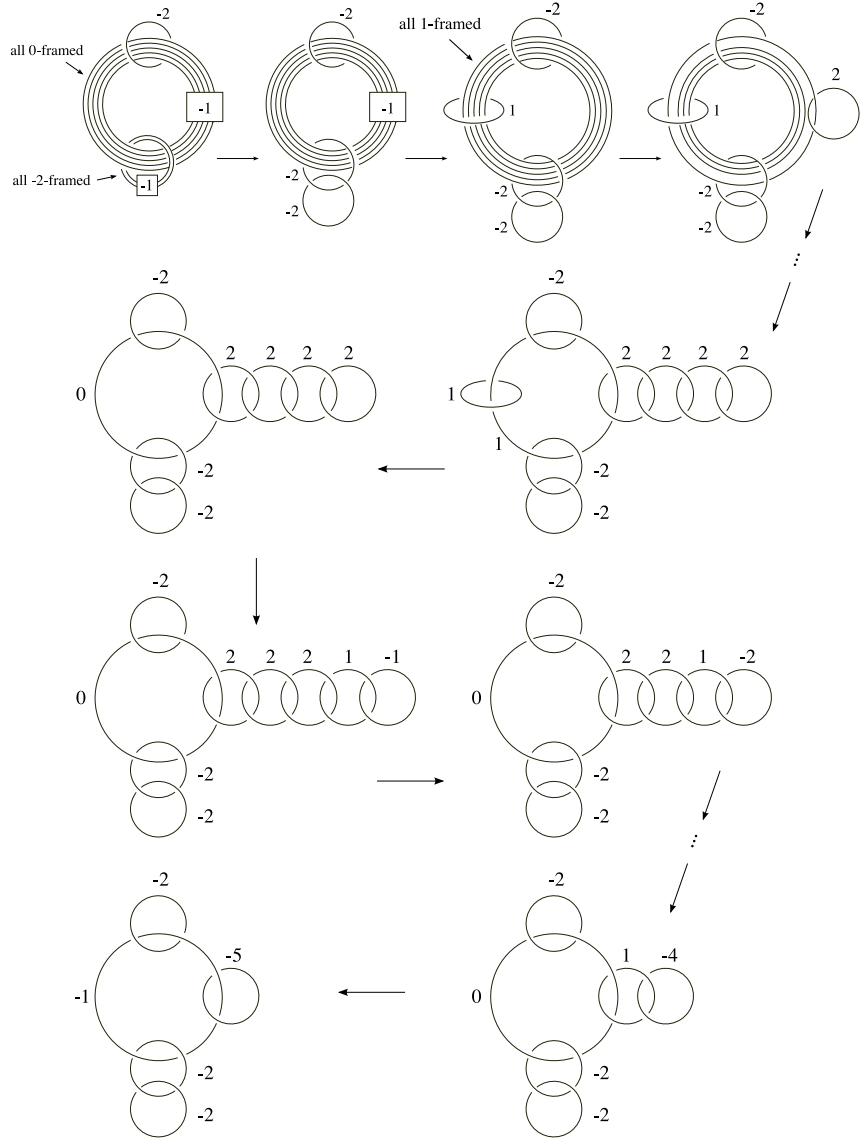}
\caption{Kirby moves.} \label{kirby125}
\end{figure} Suppose that $(X, J)$ is a Stein-filling for $\xi$, and $\spc_J$ is the corresponding $\Sc$ structure on $X$. Note that $\spc_J$ restricts to the $\Sc$ structure $\spc_{\xi}$  on $Y$, and $c_1(\spc_{\xi})=0$. Let $\bar \xi$ be the contact structure on $Y$ conjugate to $\xi$; then  $\spc_{\bar \xi} = \spc_{\xi}$, and   $\bar \xi$ has a Stein-filling $(X, -J)$ with $\spc_{-J}= \bar{\spc}_J$.  We claim that   $c_1(\spc_J)=0$. Indeed, otherwise the $\Sc$ structures $\spc_J$ and $\spc_{-J}$ are not isomorphic, and by \cite{pla3} the contact elements $c(\xi)$ and $c(\bar \xi)$ are linearly independent  elements in homology $\hf(-Y, \spc_{\xi})$. But $Y$  is an $L$-space, so $\hf(-Y, \spc_{\xi})$ has rank 1, a contradiction.

\begin{figure}[!htbp] 
\includegraphics[width = 10.6cm]{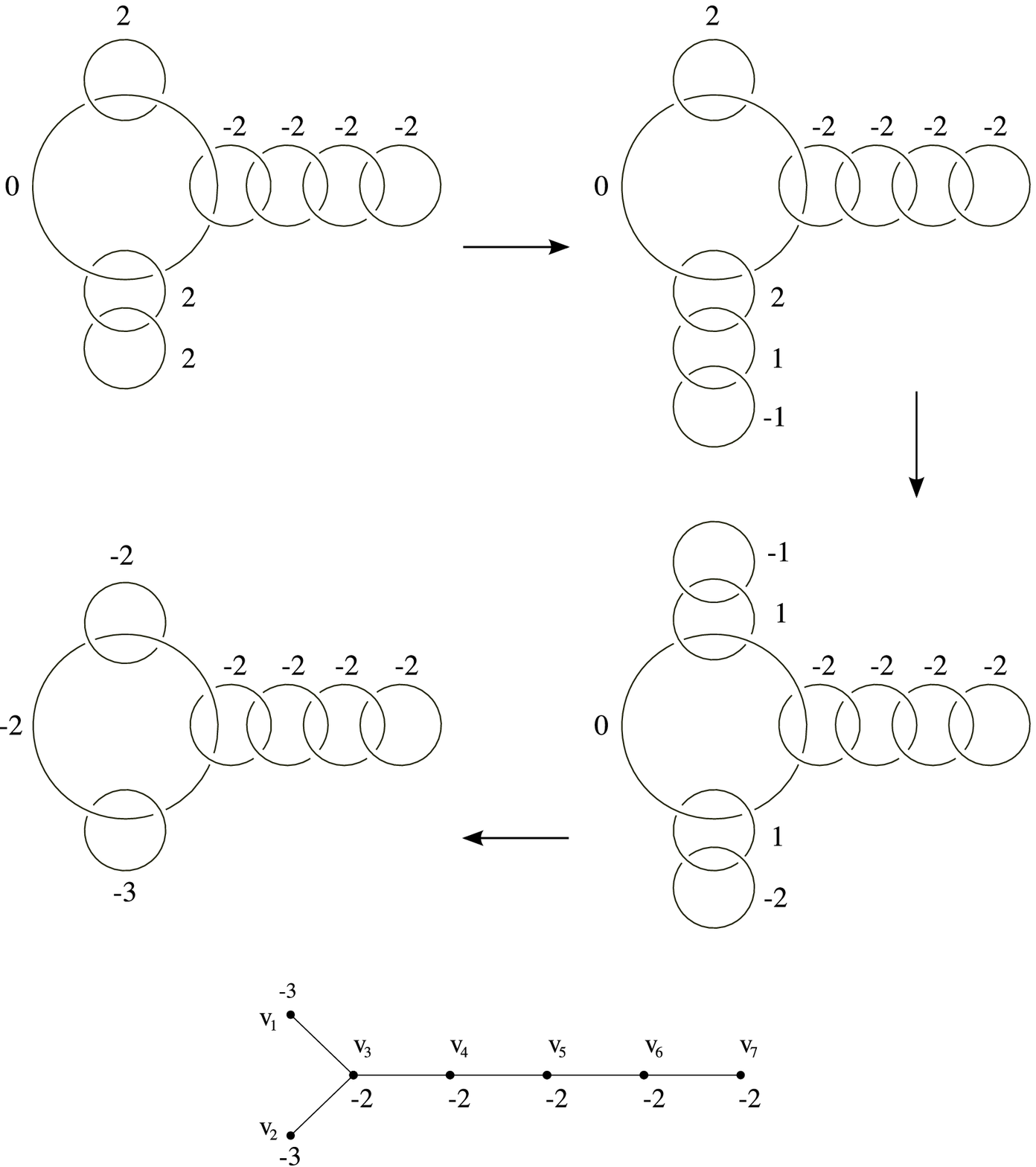}
\caption{Kirby diagrams for $-Y$ and the plumbing graph for $W$.} \label{-Yplumb}
\end{figure}  
On the other hand, the fact that $Y$ is an $L$-space implies \cite{osz2} that $b_2^+(X)=0$ for any symplectic filling $X$ of a contact 
structure on $Y$. By the argument in \cite{GLS}, it follows that $b_1(X)=0$. Now, observe that the space $-Y$
can be represented as the boundary of the plumbing shown on Figure \ref{-Yplumb}.

Denote by $W$ the 4-manifold with boundary $-Y$ given by this plumbing. If $X$ is a symplectic filling for $\xi$, 
then $X\cup W$ is an oriented negative-definite closed 4-manifold. By Donaldson's theorem, 
the intersection form on $X \cup W$ is standard diagonal $\langle -1 \rangle ^n$. To get restrictions on the intersection form 
of $X$, we consider the embeddings of the lattice given by Figure \ref{-Yplumb} into the standard negative-definite 
lattice, following \cite{Li}. Let $e_i$, $i=1, 2, \dots n$, be the basis of $ \langle -1 \rangle ^n$ such that $e_i \cdot e_j = -\delta_{ij}$.  
Let $v_i$ be the basis of  $H_2(W)$ corresponding to the vertices of the plumbing graph of Figure \ref{-Yplumb}. Up to permutations and sign reversals of $e_i$ (which are automorphisms of the lattice $ \langle -1 \rangle ^n$), we have 
\begin{align*}
&v_3 \mapsto e_1+e_2, \quad v_2 \mapsto -e_1 +e_3,  \quad v_1 \mapsto -e_1 -e_3 +e_5, \\     &v_4 \mapsto -e_2 +e_4, \quad
v_5 \mapsto -e_4 +e_6, \quad v_6 \mapsto -e_6+e_7, \quad v_7 \mapsto -e_7+e_8                                           
\end{align*}
(Another possibility would be for the first four vectors to embed as 
$$
v_3 \mapsto  e_1+e_2, \quad v_2 \mapsto -e_1 +e_3, \quad v_1 \mapsto -e_2 + e_4 + e_5, \quad v_4 \mapsto -e_1 - e_3, 
$$
but this leads to a contradiction when we try to embed $v_5$.)

The orthogonal complement $L$ of the image of the lattice generated by images of $v_i$'s in $\langle -1 \rangle ^n$ is then spanned by 
the vectors
$$
-e_1+e_2-e_3+e_4-2e_5+e_6+e_7+e_8, \quad
 e_9, \quad \dots \quad e_n,
$$
and the intersection form on $L$ is the diagonal form $ \langle -11 \rangle  \oplus \langle -1 \rangle ^{n-8}$.    
Because $H_1(Y) = \zz/11$ (indeed, $|H_1(Y)| = \det(10_{125})=11$), and both $H_2(X)$, $H_2(W)$ are torsion-free, we have 
$$
0 \to H_2(X)\oplus H_2(W) \to H_2(X\cup W) \to \zz/11 \to 0,  
$$  
and thus $H_2(X,\zz)$ is a subgroup of  $L= \zz^{n-7}$ of index $11$. Set $m=n-7=b_2(X)$, and let 
$\{u_1,\, u_2, \, \dots u_m\}$ be basis 
of $L$ in which the form is diagonal, and $u_1 \cdot u_1 = -11$. The vectors $11 u_1$,  $11 u_2$, ... $11 u_m$
lie in $H_2(X, \zz)$, and generate $H_2(X, \QQ)$ over $\QQ$.

It follows that
 $c_1(\spc_J)$ must evaluate as an odd integer on each vector   $11 u_1$, $11 u_2$, $\dots 11 u_m$; 
but since    $c_1(\spc_J)=0$, this means that  $m=0$. 
Then $d_3(\xi)=0$, which contradicts the calculation $d_3(\xi)=-\frac12$. 

\end{example}

\begin{figure}[!htbp] 
\includegraphics[width = 10.5cm]{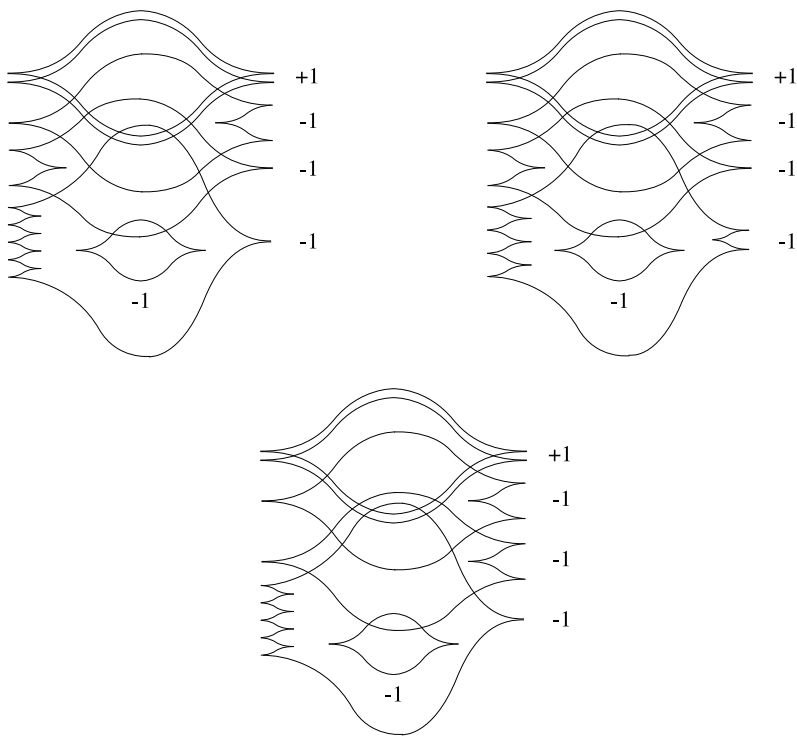}
\caption{The manifold $Y=M(-1; 2/3, 1/2, 1/5)$ carries three tight contact structures: $\xi_1$ (top left), $\xi_2$ (top right), and $\Xi$ (bottom).}\label{3cont} 
\end{figure}  

\begin{remark}
The branched double cover of the knot $K$ in the previous example is the Sefert fibered space $Y=M(-1; 2/3, 1/2, 1/5)$. 
Tight contact structures on this space were 
classified in \cite{GLS}: $Y$ carries three tight contact structures $\xi_1$, $\xi_2$ and $\Xi$ given by surgery diagrams on Figure~\ref{3cont}.   Of these,  $\xi_1$ and  $\xi_2$ are known to be Stein-fillable; 
 we can thus conclude that our contact structure $\xi$ is isotopic to $\Xi$.

\end{remark}

\begin{example}  \label{e141} Consider the transverse representative of the mirror of the knot $10_{141}$ given by the 
braid $\sigma_1^{-4} \sigma_2 \sigma_1^{3} \sigma_2^2$. We consider the family of braids  
$$
K_r = \sigma_1^{-r} \sigma_2 \sigma_1^{3} \sigma_2^2.
$$
The contact surgery description for the corresponding contact structures is given on Figure \ref{141};
the surgery diagrams are quite similar to those in the previous example, but have one extra component.
The Kirby calculus moves similar to those in Figure \ref{kirby125} show that the branched double cover 
is the Seifert fibered space $M(-1; 2/3, 2/3, 1/n)$. 
\begin{figure}[!htbp] 
\includegraphics[width = 5.5cm]{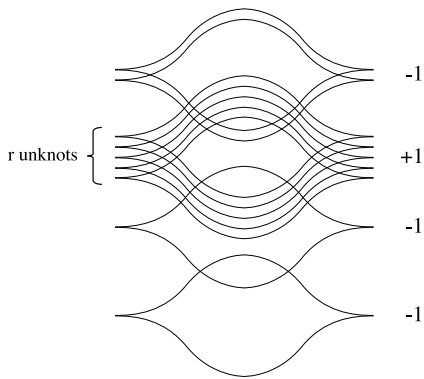}
\caption{The surgery diagrams for the branched double covers of the transverse links $K_r$ in Example \ref{e141}.}\label{141} 
\end{figure}

As before, we can show that all the braids $K_r$ are quasi-alternating. Indeed, 
we resolve one of the negative crossings to obtain $K_{r-1}$ and a trefoil as two resolutions;
we also observe that $K_0$ is the connected sum of two trefoils.
Since  $|\det (\text{trefoil})| =3$, $|\det( K_0)|=9$ and$ |\det (K_r)| = |H_1 ( M(-1; 2/3, 1/2, 1/n) )| = 9+3r$,  
each $K_r$ is quasi-alternating by induction. 

Next, we compute $sl(K_r)=3-r$, and $s=\sigma(K_r)= 4-r$; the hypotheses of Corollary \ref{cor:tightness} are therefore 
satisfied, and all branched covers $\Sigma(K_r)$ are tight contact manifolds.

For the contact structure on the branched cover of $K_4$, we compute $d_3=0$, which provides no obstruction to Stein-fillability.  However, for the braid $K_6$ we get $d_3=-\frac12$. We then argue as in the previous example 
to show that the branched cover of $K_6$ is not Stein-fillable 
(and thus the branched double covers of all braids  $K_r$ with $r\geq 6$ are not Stein-fillable either). 
\begin{figure}[!htbp] 
\includegraphics[width = 11cm]{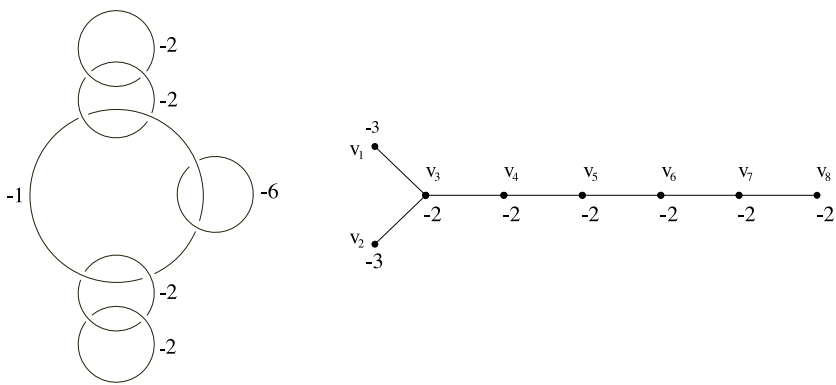}
\caption{A Kirby diagram for $Y= \Sigma(K_6)$ and the plumbing graph for $-Y$ of Example \ref{e141}.}\label{mY141} 
\end{figure}  
Denote $Y= \Sigma(K_6) = M(-1; 2/3, 1/2, 1/6)$; then $-Y$ is the boundary of the plumbing $W$ given by the graph 
on Figure \ref{mY141}.  As before, for any symplectic filling $X$ of $Y$  
the union  $X\cup W$ is a negative-definite closed 4-manifold with the  
standard diagonal intersection form. Up to changing the signs and the order of the vectors $e_i$ in the diagonal basis, there is a unique embedding
 of the lattice given by Figure \ref{mY141} into 
$\langle -1 \rangle ^n$, given by 
\begin{align*}
&v_3 \mapsto e_1+e_2, & &v_1 \mapsto -e_1 -e_3 +e_5, & &v_2 \mapsto -e_1 +e_3+e_4, & &v_4 \mapsto -e_2 +e_6,\\      
&v_5 \mapsto -e_6 +e_7, & &v_6 \mapsto -e_7+e_8,    & &v_7 \mapsto -e_8+e_9, & &v_8 \mapsto -e_9+e_{10},   
\end{align*}
and thus the orthogonal complement of this lattice in $\langle -1 \rangle ^n$ is 
$ \langle -9 \rangle  \oplus \langle -3 \rangle \oplus \langle -1 \rangle ^{n-10}$.
As in the previous example, we must have $c_1(X)=0$ for any Stein-filling.  
Since $|H_1(Y)| = 27$, similar parity argument  
shows that $b_2(X)=0$, and so $d_3$ must be zero, a contradiction. 

\end{example}

\begin{remark}

One can try to argue as in \cite{GLS} to investigate symplectic fillability in
Examples \ref{ex:e125} and \ref{e141}: a slightly more involved
agrument modulo 8 puts further restrictions on the value $d_3$ for symplectic fillings 
(with diagonal odd intersection form). However, this gives no obstruction to symplectic fillability 
of any contact structures in the above two examples. 

In the opposite direction, certain tight open books with the punctured torus page and pseudo-Anosov 
monodromy can be shown 
to be symplectically fillable as perturbations of taut foliations \cite{hkm2}.  
We note that our examples are not pseudo-Anosov, so these results do not apply.

\end{remark}

\begin{example} \label{e130}
A transverse representative of the mirror of $10_{130}$ with the maximal self-linking number 
is given by the braid $\sigma_1^{-3} \sigma_2 \sigma_1^2 \sigma_2^2 \sigma_3 \sigma_2^{-1} \sigma_3$.
We consider a family of transverse braids 
 $$
  K_r=\sigma_1^{-r} \sigma_2 \sigma_1^2 \sigma_2^2 \sigma_3 \sigma_2^{-1} \sigma_3.
  $$
First, we check that all the underlying links are quasi-alternating. Resolve of the negative crossings  
among those given by $\sigma_1^{-r}$ to obtain $K_{r-1}$ as one of the resolutions and the unknot as the other.
Observe that $K_0$ is a two-component alternating link of $\det =14$ (with $5_2$ knot and the unknot as 
components, linked once). Finally, compute $|\det (K_r)| = 14+r$ (one way to see this is to compute the size of $H_1$
of the branched double cover of $K_r$ which is a Seifert fibered space shown on Figure \ref{130}). 

\begin{figure}[!htbp] 
\includegraphics[width = 11cm]{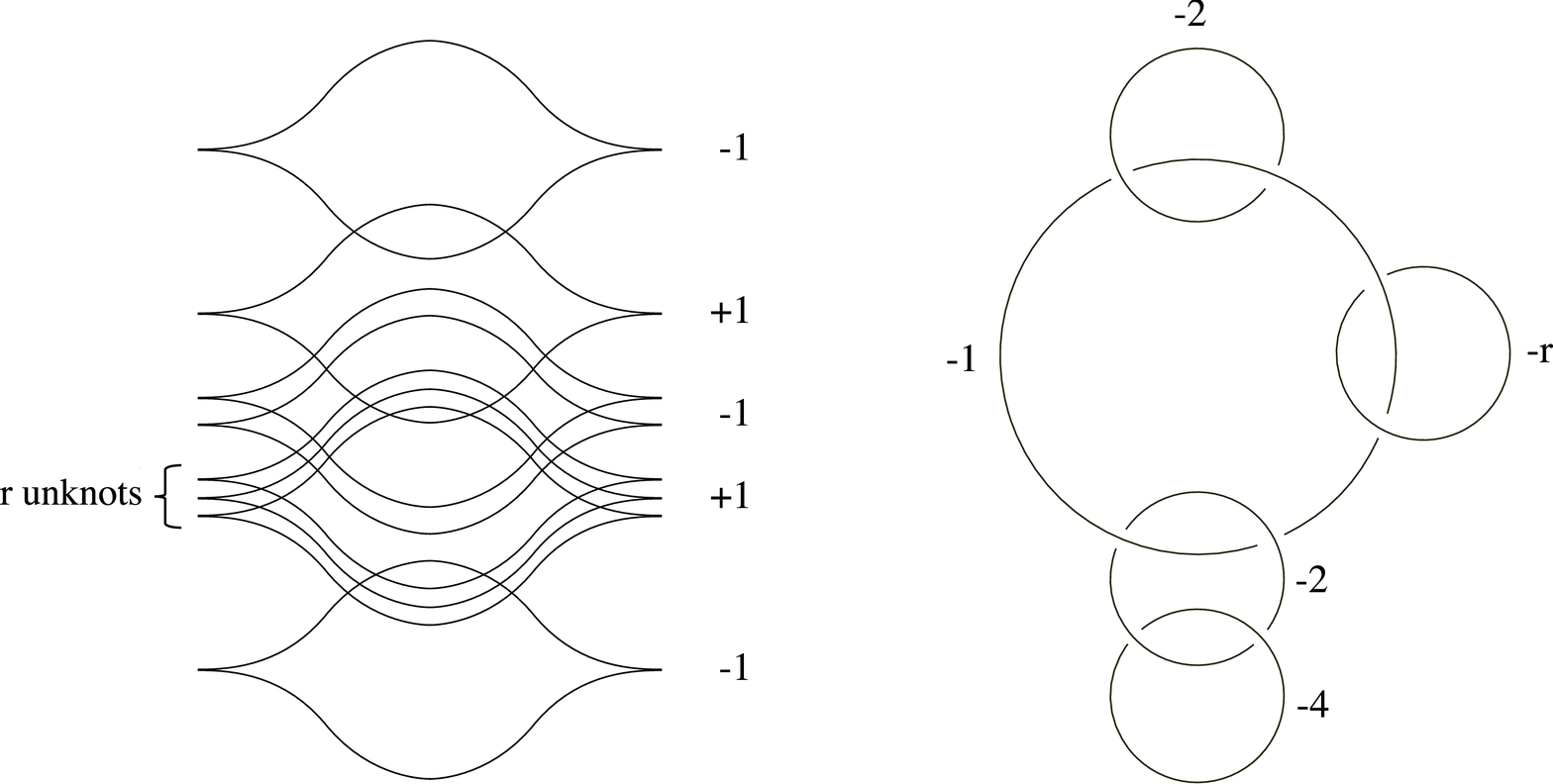}
\caption{The surgery diagrams for the branched double covers of the transverse links $K_r$ in Example \ref{e130}.}\label{130} 
\end{figure}
  
The hypothesis of Corollary \ref{cor:tightness} holds: $sl(K_r)= 2-r = \sigma-1$. Therefore, the branched double
covers of the transverse links $K_r$ are all tight.   

These contact structures are also Stein non-fillable for $r\geq 5$; this can be seen by using the arguments 
of \cite[Section 7]{bald4}. (Although Theorem 7.1 of \cite{bald4} is stated for the case of 3-braids, its proof 
carries over to our example, since the corresponding 3-manifold is an $L$-space, the contact structure is self-conjugate, and its  3-dimensional invariant 
is negative. We leave the details to the reader.)  

\end{example}

\begin{remark} In Examples \ref{ex:e125} and \ref{e141}, transverse links are 3-braids, and 
the contact structures on the branched double covers can be given by open books whose page 
is a once-punctured torus. Tightness of these contact structures can be established by using results 
of \cite{hkm3} or \cite{bald1}. However, these results do not apply to Example \ref{e130} as it deals with 4-braids, 
and the page of the corresponding open books is a twice-punctured torus.  

In all of the above examples, we  checked explicitly that our families of links are quasi-alternating. In fact, a more 
involved argument (see \cite{bald4}) shows that all braids of the form 
$(\sigma_1 \sigma_2)^3   \sigma_1 \sigma_2^{-a_1} \dots  \sigma_1 \sigma_2^{-a_n}$ (with $a_i\geq 0$) are quasi-alternating and 
have $sl=\sigma-1$. Thus Corollary \ref{cor:tightness} applies to many more knots; moreover, a lot of the
corresponding contact manifolds are not Stein-fillable. 

We also note that a weaker condition, $\text{rk} (\kh(K)) = |\det (K)|$ is sufficient to ensure that the spectral 
sequence from $\kh$ to $\hf$ collapses at the $E^2$ stage. 
This can be checked (for  any individual reasonably small knot)
by a computer, for example using {\tt Kh} program that computes the rank of reduced Khovanov 
homology with $\zzt$ coefficients. Checking the second condition, $sl(K)=s-1$, is also routine for $\kh$-thin
knots (alternatively, one can use the {\tt Trans} program to check $\psi\neq 0$). Proving  tightness of the contact 
structure on the branched double cover is thus reduced to a computer calculation. 

\end{remark}

In the next example, we prove tightness of a contact structure which does not obviously arise as the branched cover of $(S^3,\xi_{st})$ along a transverse link (the corresponding open book includes a Dehn twist around $\alpha_0$).



\begin{example} \label{one-shark} Consider the contact manifold with surgery diagram shown on Figure \ref{shark1}. 
This diagram corresponds (via the procedure described in the beginning of this section together with 
a Legendrian ``flip'' as in Example \ref{ex:e125}) to the open book whose page $S$ is a genus 2 surface with one boundary component, and the monodromy is 
$\phi = D_{\alpha_1} D_{\alpha_2} D_{\alpha_3} D_{\alpha_4} (D_{\alpha_3})^{-5} (D_{\alpha_4})^2 D_{\alpha_0}$ (where $D_{\alpha_1}$ is performed first and $D_{\alpha_0}$ last).  
Using  {\tt Trans} and {\tt Kh} programs, we find that $\psi(S, \psi) \neq 0$ and that $\text{rk} ( \kh(S, \phi)) = 9$. 
Since $\text{rk}( H_1(Y)) =9$ as well (we compute it as the determinant of the linking matrix), 
the spectral sequence collapses at the $E^2$ term. Therefore,  Figure \ref{one-shark} represents a tight contact structure. 
(A few Kirby moves as above show that the underlying smooth manifold is the Seifert fibered space 
$Y=M(-1; 2/3, 1/3, 1/5)$).

\begin{figure}[!htbp] 
\includegraphics[width = 4.5cm]{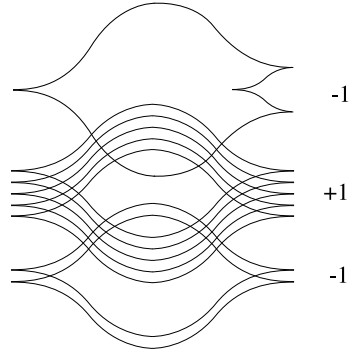}
\caption{A surgery diagram for the tight contact structure in Example \ref{one-shark}.}\label{shark1} 
\end{figure}

\end{example}

\subsection{Some limitations}
The examples in this subsection illustrate some of the limitations of our method.

\begin{example} \label{nex}
\begin{figure}[!htbp] 
\includegraphics[width = 11cm]{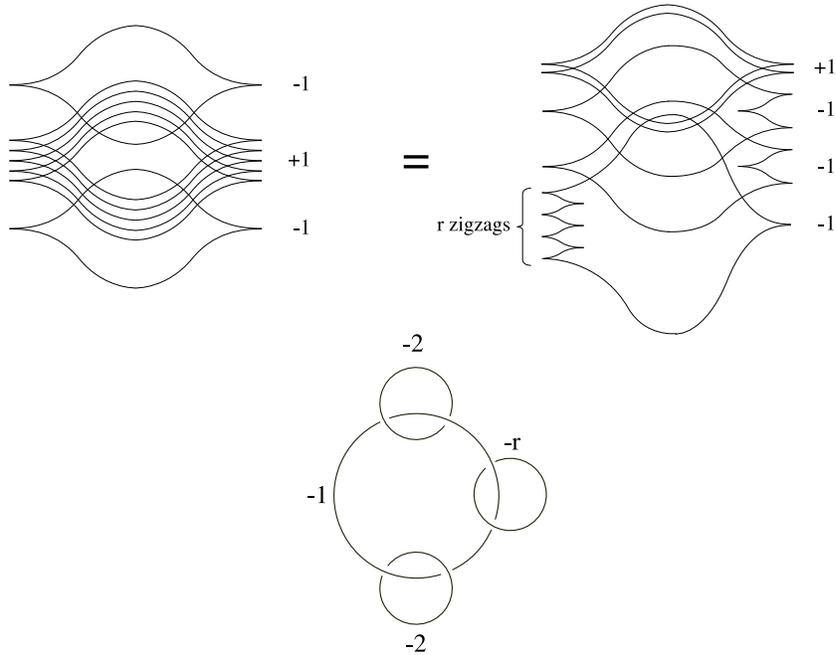}
\caption{The surgery diagrams for the branched double covers of the transverse links in Example \ref{nex}.}\label{non-ex} 
\end{figure}  
 
Consider a family of  (non-quasi-alternating) 
transverse 3-braids 
$$
K_r=\sigma_1^{-r} \sigma_2 \sigma_1^2  \sigma_2.
$$

Using the  {\tt Trans} program and \cite[Theorem 4]{pla1}, we see that the transverse invariant $\psi$ vanishes  for $r>2$. However, 
the calculations of the contact invariant in \cite{hkm3} imply that the corresponding contact structures on the branched double covers are all tight, and have $c(\xi)\neq 0$. Thus, Khovanov homology fails to detect tightness in this case. It is interesting to take a closer look at the corresponding contact structures: they are given by surgery diagrams on the left of Figure \ref{non-ex}, and are very similar to the contact structures from Example \ref{ex:e125}. As the latter are obtained from the former by Legendrian surgery on a knot, the  contact structures $\xi_{K_r}$ cannot be Stein-fillable. The underlying smooth manifold is $M(-1;1/2, 1/2, 1/r)$; by \cite{GLS} it carries a unique tight, Stein non-fillable contact structure for each $r$ (shown on the right of Figure \ref{non-ex}), and we conclude that the contact structures $\xi_{K_r}$ are precisely those considered in \cite{GLS}.  By \cite{GLS}, most of these contact structures are not symplectically fillable, and one may wonder whether there is a relationship between the vanishing of $\psi$ and symplectic non-fillability which goes beyond Proposition \ref{prop:vanishing} (although such a relationship seems quite improbable). 
    
\end{example}

\begin{example} \label{nex2} We find it instructive to give an example where a non-vanishing element 
$\psi\in \kh(S, \phi)$ dies in the spectral sequence, so that $c(\xi)=0$. Consider the open book  with once-punctured genus two page $S$, and the monodromy $\phi = D_{\alpha_1} D_{\alpha_2} D_{\alpha_3} D_{\alpha_4} (D_{\alpha_3})^{-5} D_{\alpha_4} D_{\alpha_0}$; the corresponding contact manifold is given by the surgery diagram on Figure \ref{poincare}. By Kirby calculus similar to Figures \ref{kirby125}, \ref{-Yplumb}, the underlying manifold is the boundary of the $E_8$ plumbing, i.e. the Poincar\'e homology sphere $-\Sigma(2, 3, 5)$ with orientation reversed; by \cite{EH}, it carries no tight contact structures, therefore $c(\xi)=0$. However, {\tt Trans} tells us that $\psi(S, \phi) \neq 0$.  (In this case we have $\text{rk} ( \hf) =1$, while the $E^2$ term of the spectral sequence is of course bigger: {\tt Kh}  computes $\text{rk} ( \kh(S, \phi))   = 7$.) This example shows that the vanishing of $\khc(S,\phi)$ is not an invariant of $\xi_{S,\phi}$, as discussed at the end of Section \ref{sec:compute}.

\begin{figure}[!htbp] 
\includegraphics[width = 11cm]{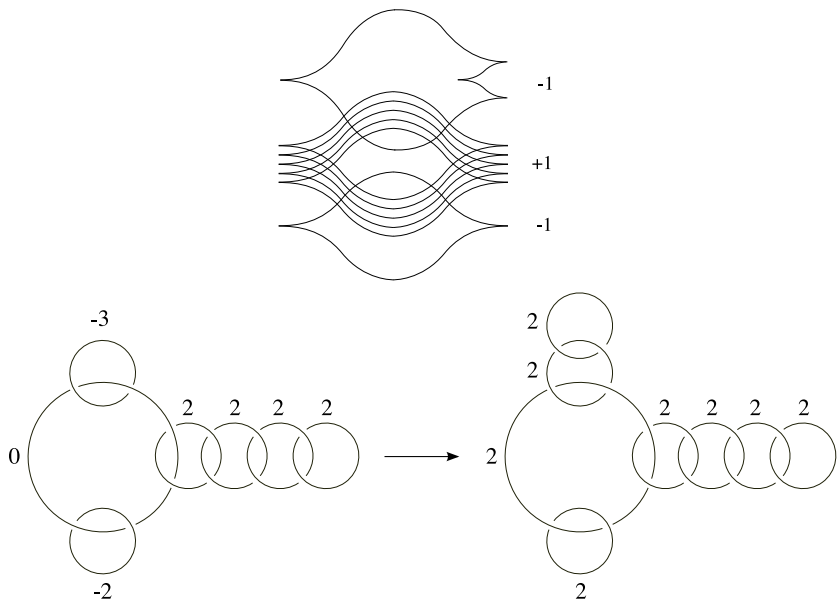}
\caption{A surgery diagram for the contact structure in   Example \ref{nex2}.}\label{poincare} 
\end{figure}

\end{example}

\subsection{Non-fillability by means of Khovanov homology}
\label{ssec:nonfill}

In this section, we give two examples which illustrate the use of Proposition \ref{prop:vanishing} in proving that a contact structure is not strongly symplectically fillable.

\begin{example}
Let $m(10_{132})$ denote the mirror of the knot $10_{132}$. The Poincar{\'e} polynomial for the reduced Khovanov homology of $m(10_{132})$ is $$t^{-7}q^{-14} + t^{-6}q^{-12} + t^{-5}q^{-10} + 2t^{-4}q^{-8} + t^{-3}q^{-8} + t^{-3}q^{-6} + t^{-2}q^{-6} + t^{-2}q^{-4} + t^{-1}q^{-2} + t^{0}q^{-2}.$$ (Here, the exponent of $t$ keeps track of the homological grading, while the exponent of $q$ keeps track of the quantum grading.) Note that the reduced Khovanov homology is supported in non-positive homological gradings. Therefore, Corollary \ref{cor:vanishing} implies that if $K$ is any transverse knot which is smoothly isotopic to $m(10_{132})$, and $sl(K)+1 \neq -2$, then $c(\xi_K) = 0$, and, hence, $\xi_K$ is not strongly symplectically fillable (or weakly symplectically fillable, since $\Sigma(K)$ is a rational homology 3-sphere\footnote{Strong symplectic fillability of $(M,\xi)$ is equivalent to weak symplectic fillability when $M$ is a rational homology 3-sphere \cite{oono}.}
). For example, the 4-braid $$ \sigma_3^{-1} \sigma_2^{2} \sigma_3^{-2} \sigma_2^{-1} \sigma_3 \sigma_1 \sigma_2^{-1} \sigma_1^{-2}$$ corresponds to a transverse representative $K$ of $m(10_{132})$ (this braid representative comes from \cite{kng}) with $sl(K) +1 = - 6$; hence, $c(\xi_K)=0$.

\end{example}

\begin{remark}
It is not hard to find transverse knots which satisfy the hypotheses of Corollary \ref{cor:vanishing}. For instance, of the 49 prime knots with 9 crossings, the knots (or their mirrors) in the set below have reduced Khovanov homologies supported in non-positive homological gradings: $$\{9_1,\,9_2,\,9_3,\,9_4,\,9_5,\,9_6,\,9_7,\,9_9,\,9_{10},\,9_{13},\,9_{16},\,9_{18},\,9_{23},\,9_{35},\,9_{38},\,9_{45},\,9_{46},\,9_{49}\}$$
\end{remark}

In the next example, we prove non-fillability of a contact structure which does not obviously arise as the branched cover of $(S^3,\xi_{st})$ along a transverse link (the corresponding open book includes a Dehn twist around $\alpha_0$).

\begin{example}
Let $(S,\phi)$ be the abstract open book $$(S_{2,1}, D_{\alpha_1}^{2} D_{\alpha_2}^{-1} D_{\alpha_3}^{-1} D_{\alpha_4}^{-1} D_{\alpha_0} D_{\alpha_4}^{-1} D_{\alpha_3}^{-1} D_{\alpha_2}^{-1} D_{\alpha_1}^{-1} D_{\alpha_3} D_{\alpha_4} D_{\alpha_2}).$$ Using the program {\tt Kh}, we find that the Poincar{\'e} polynomial for $\kh(S,\phi)$ is $$t^{-6} + 2t^{-5} + t^{-4} + t^{-3} + 2t^{-2} + 2 t^{-1}.$$ (Here, the exponent of $t$ keeps track of the homological grading.) Note that $\kh(S,\phi)$ is supported in negative homological gradings. Since the homological grading of $\khc(S,\phi)$ is zero, it follows at once that $\khc(S,\phi)$ vanishes. Therefore, Proposition \ref{prop:vanishing} implies that $c(S,\phi) = 0$, and, hence, that $\xi_{S,\phi}$ is not strongly symplectically fillable.

\end{example}

\newpage

\section{Comultiplication}
\label{sec:comult}
In \cite{bald3}, the first author shows that the contact invariant $c(S,\phi)$ satisfies the following naturality property with respect to a comultiplication map on Heegaard Floer homology. 

\begin{theorem}[{\rm \cite[Theorem 1.4]{bald3}}]
\label{thm:comultc}
If $h$ and $g$ are two boundary-fixing diffeomorphisms of $S$, then there is a comultiplication map $$\mu:\hf(-M_{S, h g}) \rightarrow \hf(-M_{S, h}) \otimes_{\zzt} \hf(-M_{S,g})$$ which sends $c(S, h g)$ to $c(S, h)\otimes c(S,g)$.
\end{theorem}

In this section, we use comultiplication to strengthen the relationship between $\khc(S,\phi)$ and $c(S,\phi)$. Our main result is the following.

\begin{theorem}
\label{thm:comult}
If $h $ and $g$ are two compositions of Dehn twists around curves in $S$, then there is a comultiplication map $$\mu: \kh(S, h g) \rightarrow \kh(S, h) \otimes_{\zzt}\kh(S,g)$$ which sends $\khc(S, h g)$ to $\khc(S, h)\otimes\khc(S,g)$.
\end{theorem}

\begin{proof}[Proof of Theorem \ref{thm:comult}]
Suppose that $h =D_{\gamma_1}^{\epsilon_1} \cdots D_{\gamma_p}^{\epsilon_p}$ and $g =D_{\gamma_{p+1}}^{\epsilon_{p+1}} \cdots D_{\gamma_{n}}^{\epsilon_{n}}$, and let $$L_1 = \bigcup_{j=1}^p \gamma_j \times \{t_j\}\text{,  and  }L_2= \bigcup_{j=p+1}^n \gamma_j \times \{t_j\}$$ be the corresponding links in $-M_{S,id}$. Recall that $-M_{S,hg}$ is obtained from $-M_{S,id}$ by performing $\epsilon_j$-surgery on each $\gamma_j \times \{t_j\}$. Let $b_1,\dots,b_{2k+r-1}$ be the arcs on $S=S_{k,r}$ indicated in Figure \ref{fig:Sarcs}. Choose two points, $0<T_A<t_1$ and $t_p < T_B < t_{p+1}$, and let $\beta_j$ be the knot in $-M_{S,id}$ defined by $$\beta_j = b_j\times \{T_A\} \,\cup\, b_j \times \{T_B\}\, \cup\, \partial b_j \times [T_A,T_B].$$

\begin{figure}[!htbp]
\begin{center}
\includegraphics[width=11cm]{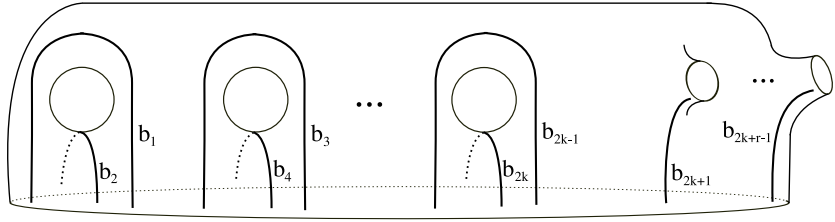}
\caption{\quad The endpoints of the arcs $b_j$ lie in a small collar neighborhood of $\partial S_{k,r}$ which is disjoint from the curves $\gamma_j$.}
\label{fig:Sarcs}
\end{center}
\end{figure}

Denote by $Y$ the 3-manifold obtained from $-M_{S,id}$ by performing $0$-surgery on each $\beta_j$, and let $L_3 \subset Y$ be the induced link $$L_3=\bigcup_{j=1}^n \gamma_j \times \{t_j\}.$$ Let $(C(S,hg)_0,D_0)$ be the complex associated to the multi-diagram compatible with all possible combinations of $\epsilon_j$-, $0$-, and $\infty$-surgeries on the components of $L_3$. If $(S,hg)_{i,0}$ is the result of $0$-surgery on each induced $\beta_j$ in $(S,hg)_i$, then $$C(S,hg)_0=\bigoplus_{i \in \{0,1\}^n} \cf((S,hg)_{i,0}).$$ We define the $I$-grading on $C(S,hg)_0$ by $I(x) = |i|$ for $x \in \cf((S,hg)_{i,0}).$ By the construction in \cite{osz12}, there is an $I$-filtered chain map $$F:C(S,hg)\rightarrow C(S,hg)_0,$$ which is a sum of maps $$F_{i,i'}:\cf((S,hg)_i)\rightarrow \cf((S,hg)_{i,0}),$$ over all pairs $i,i'$ for which $i \leq i'$.

Observe that $Y$ is diffeomorphic to $-M_{S,id}\,\#-M_{S,id}$. If we think of the links $L_1$ and $L_2$ as being contained in the first and second $-M_{S,id}$ summands of $Y$, respectively (so that there is no linking between $L_1$ and $L_2$), then $L_3$ is simply the union $L_1 \cup L_2$. This becomes evident after a sequence of handleslides (see Figure \ref{fig:comultdbc} for an example). In particular, $$(S,hg)_{i,0} \cong (S,h)_{i_H} \# (S,g)_{i_G},$$ where $i_H = (i_1,\dots,i_p)$ and $i_G = (i_{p+1},\dots,i_n).$ Therefore, the complex $E^1_I(S,hg)_0$ decomposes as a tensor product, \begin{equation}\label{eqn:tensorhg}E^1_I(S, h g)_0 \cong E^1_I(S, h) \otimes_{\zzt}E^1_I(S,g).\end{equation}

\begin{figure}[!htbp]
\begin{center}
\includegraphics[width=11cm]{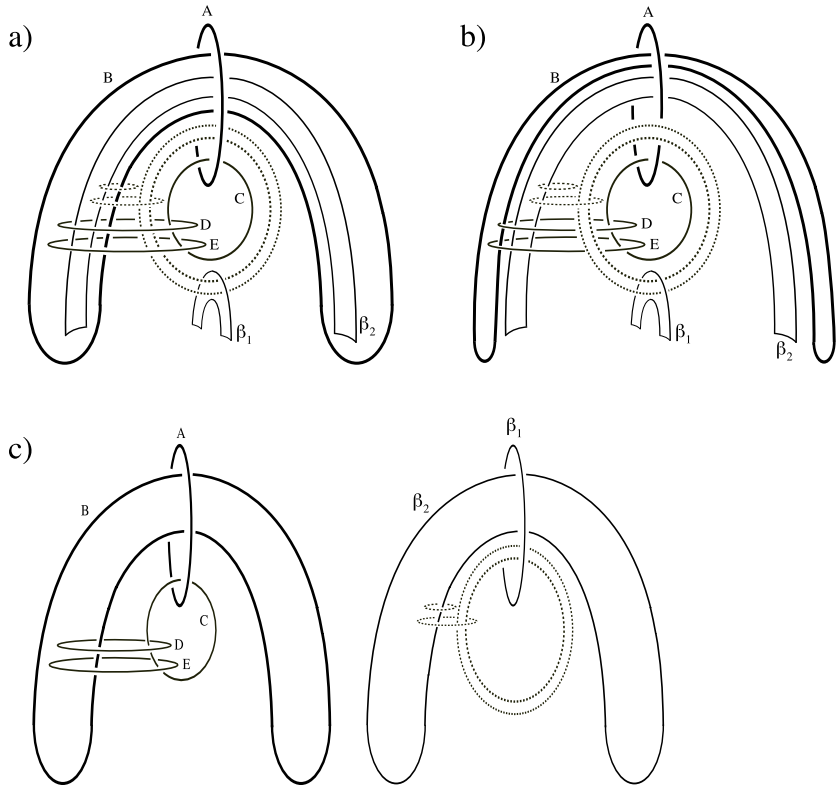}
\caption{\quad In this example, $S=S_{1,1}$, $h = D_{\alpha_1}D_{\alpha_2}^{-2}D_{\alpha_1}^{-1}$, and $g = D_{\alpha_2}D_{\alpha_1}^2$. We begin with the diagram in a). Though we have not written them down, the surgery coefficients on $A$ and $B$ are both $0$. The dotted curves comprise the link $L_1$ corresponding to the diffeomorphism $h$, and the solid curves $C$, $D$, and $E$ comprise the link $L_2$ corresponding to $g$. $L_3$ is the union of $L_1$ with $L_2$, viewed as a link in the 3-manifold $Y$ obtained by performing $0$-surgeries on $\beta_1$ and $\beta_2$. In $Y$, slide $B$ over $\beta_2$ so that $B$ no longer links the components of $L_1$. Next, slide $A$, $D$, and $E$ over $\beta_1$ so that they no longer link the components of $L_1$. We arrive at the diagram in b), which is isotopic to the diagram in c). Note that $L_1$ and $L_2$ are unlinked in $Y$.}
\label{fig:comultdbc}
\end{center}
\end{figure}

$F$ induces a map $$F^1: E^1_I(S,hg) \rightarrow E^1_I(S,hg)_0$$ which is the sum of the maps $$(F_{i,i})_*:\hf((S,hg)_i)\rightarrow \hf((S,hg)_{i,0}),$$ over all $i\in \{0,1\}^n$. Each $(F_{i,i})_*$ is induced by the 2-handle cobordism corresponding to the $0$-surgeries on the induced knots $\beta_j \subset (S,hg)_i$. Note that $$(S,hg)_{i_o}\cong \#^{2k+r-1}(S^1 \times S^2),$$ and the element $\khct(S,hg)$ is the generator of $\wedge^{2k+r-1}H_1((S,hg)_{i_o}) \cong \zzt$ under the identification of $\hf((S,hg)_{i_o})$ with $\wedge^*H_1((S,hg)_{i_o}).$ Similarly, $$(S,hg)_{i_o,0}\cong \#^{4k+2r-2}(S^1 \times S^2),$$ so $\hf((S,hg)_{i_o,0})$ may be identified with $\wedge^*H_1((S,hg)_{i_o,0}).$ Since the induced knots $\beta_j\subset (S,hg)_{i_o}$ are unknots, $(F_{i_o,i_o})_*$ sends $\khct(S,hg)$ to the generator of $\wedge^{4k+2r-2}H_1((S,hg)_{i_o,0})$, by Proposition \ref{prop:s1s2}. Under the isomorphism in Equation \ref{eqn:tensorhg}, this generator corresponds to $$\khct(S,h)\otimes \khct(S,g) \in \cf((S,h)_{(i_o)_H}) \otimes \cf((S,g)_{(i_o)_G}).$$ The map $\mu$ induced on $E^2_I$ therefore takes $\khc(S,hg)$ to $\khc(S,h)\otimes \khc(S,g)$.

\end{proof}

\newpage

\bibliographystyle{hplain}
\bibliography{References}

\end{document}